\documentclass[11pt,reqno]{amsart}

\usepackage{amsmath,amssymb,amsfonts,amsthm,latexsym,graphicx,multirow,booktabs,
float,enumerate,tikz,lmodern,caption}
\usepackage{hyperref}
\usepackage[utf8]{inputenc}
\usetikzlibrary{calc}
\usetikzlibrary{decorations.pathreplacing}
\graphicspath{{figures/}}

\newcommand{\Al}{\mathrm{A}}
\newcommand{\Sy}{\mathrm{S}}
\newcommand{\Aut}{\mathrm{Aut}}

\newcommand{\Out}{\mathrm{Out}}

\newcommand{\B}{\textup{B}}
\newcommand{\N}{\textbf{\textup{N}}}
\newcommand{\C}{\textbf{\textup{C}}}
\newcommand{\Soc}{\mathrm{Soc}}
\newcommand{\Core}{\mathrm{Core}}

\newcommand{\ord}{\mathrm{ord}}
\newcommand{\AGL}{\mathrm{AGL}}

\newcommand{\PSL}{\mathrm{PSL}}
\newcommand{\PGL}{\mathrm{PGL}}
\newcommand{\PGammaL}{\mathrm{P\Gamma L}}

\newcommand\Sp{\mathrm{Sp}}
\newcommand\GSp{\mathrm{CSp}}
\newcommand\PSp{\mathrm{PSp}}
\newcommand\PGSp{\mathrm{PCSp}}
\newcommand\CGammaSp{\mathrm{C\Gamma Sp}}
\newcommand\PCGammaSp{\mathrm{PC\Gamma Sp}}

\newcommand\PSU{\mathrm{PSU}}
\newcommand\PGU{\mathrm{PGU}}

\newcommand\Om{\mathrm{\Omega}}
\newcommand\POm{\mathrm{P\Omega}}
\newcommand\OO{\mathrm{O}}

\newcommand\BB{\mathrm{B}}
\newcommand\D{\mathrm{D}}
\newcommand\E{\mathrm{E}}
\newcommand\G{\mathrm{G}}
\newcommand\F{\mathrm{F}}
\newcommand\J{\mathrm{J}}
\newcommand\HS{\mathrm{HS}}
\newcommand{\M}{\textup{M}}

\newtheorem{theorem}{Theorem}[section]
\newtheorem{lemma}[theorem]{Lemma}
\newtheorem{proposition}[theorem]{Proposition}

\theoremstyle{definition}
\newtheorem{definition}[theorem]{Definition}

\newtheorem{example}[theorem]{Example}
\newtheorem{remark}[theorem]{Remark}

\begin{document}

\title[On biprimitive semisymmetric graphs]{On biprimitive semisymmetric graphs}

\author{Yunsong Gan}
\address{(Gan) School of Mathematics and Statistics\\Central South University\\Changsha, Hunan, 410083\\P.R. China}
\email{songsirr@126.com}

\author{Weijun Liu}
\address{(Liu) School of Mathematics and Statistics\\Central South University\\Changsha, Hunan, 410083\\P.R. China}
\email{wjliu6210@126.com}

\author{Binzhou Xia}
\address{(Xia) School of Mathematics and Statistics\\The University of Melbourne\\Parkville, VIC 3010\\Australia}
\email{binzhoux@unimelb.edu.au}


\begin{abstract}
A regular bipartite graph $\Gamma$ is called semisymmetric if its full automorphism group $\Aut(\Gamma)$ acts transitively on the edge set but not on the vertex set. For a subgroup $G$ of $\Aut(\Gamma)$ that stabilizes the biparts of $\Gamma$, we say that $\Gamma$ is $G$-biprimitive if $G$ acts primitively on each part. In this paper, we first provide a method to construct infinite families of biprimitive semisymmetric graphs admitting almost simple groups. With the aid of this result, a classification of $G$-biprimitive semisymmetric graphs is obtained for $G=\Al_n$ or $\Sy_n$.
In pursuit of this goal, we determine all pairs of maximal subgroups of $\Al_n$ or $\Sy_n$ with the same order and all pairs of almost simple groups of the same order.

\textit{Key words:} semisymmetric graphs; primitive groups; permutation representations

\textit{MSC2020:} 20B25, 20B15, 05E18
\end{abstract}

\maketitle

\section{Introduction}\label{SEC1}

All graphs considered in this paper are finite, undirected and simple. A graph $\Gamma$ is said to be \emph{vertex-transitive} or \emph{edge-transitive}, respectively, if its full automorphism group $\Aut(\Gamma)$ acts transitively on the vertex set or edge set, respectively. We call $\Gamma$ \emph{semisymmetric}, if it is regular and edge-transitive but not vertex-transitive. The study of semisymmetric graphs is initiated by Folkman. In~1967, he constructed some infinite families of such graphs and proposed eight open problems~\cite{Folkman1967}, nearly all of which have now been resolved (see~\cite[Section~6]{CV2019} for a summary). For more results on semisymmetric graphs, we refer the reader to~\cite{Bouwer1972,CMMP2006,Ivanov1987,LX2002,MP2002,WG2021,Wilson2003}.

Let $\Gamma$ be a bipartite graph, and let $G$ be a subgroup of $\Aut(\Gamma)$ which stabilizes the biparts of $\Gamma$. We say that $\Gamma$ is \emph{$G$-biprimitive}, if $G$ acts primitively on each of the biparts of $\Gamma$. When not emphasizing $G$, we simply call $\Gamma$ \emph{biprimitive}. The first biprimitive semisymmetric graphs were given in~1985 by Iofinova and Ivanov~\cite{II1985}, as examples of such graphs of valency $3$. In 1999, Du and Maru\v{s}i\v{c}~\cite{DM19991} constructed an infinite family of $G$-biprimitive semisymmetric graphs with $G=\PSL_2(p)$, where $p$ is a prime such that $p\equiv\pm1\ (\bmod\ 8)$.

A group $G$ is said to be \emph{almost simple} with \emph{socle} $\Soc(G)=T$ if $T\leq G\leq\Aut(T)$ for some nonabelian simple group $T$. The above mentioned examples of $G$-biprimitive semisymmetric graphs by Iofinova and Ivanov~\cite{II1985} and by Du and Maru\v{s}i\v{c}~\cite{DM19991} all have $G$ almost simple. In fact, to the best of our knowledge, all known $G$-biprimitive semisymmetric graphs so far are in the case when $G$ is almost simple. The aim of this paper is to study $G$-biprimitive semisymmetric graphs with almost simple $G$.

Our first main result is a group-theoretic construction of biprimitive semisymmetric graphs. Let $L$ and $R$ be maximal subgroups of a group $G$, and let $D$ be a union of double cosets of $R$ and $L$ in $G$. The \emph{bi-coset graph} $\Gamma=\B(G,L,R,D)$ (see Definition~\ref{DEF001}) is bipartite, and the right multiplication of $G$ induces a subgroup of $\Aut(\Gamma)$ that is biprimitive and edge-transitive but not vertex-transitive. In general, it is not easy to determine whether $\Aut(\Gamma)$ is vertex-transitive or not, which is a primary obstacle in constructing biprimitive semisymmetric graphs. In light of this, we provide in the following theorem a sufficient condition for $\B(G,L,R,D)$ to be semisymmetric.

\begin{theorem}\label{THM001}
Let $G$ be an almost simple primitive group with socle $T$ and degree $n$, let $H$ be a point stabilizer in $G$, and let $g\in G$. Suppose that the triple $(G,H,g)$ satisfies the conditions:
\begin{enumerate}[{\rm(a)}]
\item\label{THM001.2} $HgH\neq Hg^{-1}H$;
\item\label{THM001.1} except for $\Al_n$ and $\Sy_n$, each overgroup of $G$ in $\Sy_n$ has socle $T$;
\item\label{THM001.3} $\Aut(T)$ has a core-free maximal subgroup $K$ containing $H$ such that $\Aut(T)=G(K\cap K^{g})$.
\end{enumerate}
Then the bi-coset graph $\B(G,H,H,HgH)$ is semisymmetric with the full automorphism group $\Aut(T)$.
\end{theorem}

\begin{remark}
According to Lemma~\ref{LEM003}\eqref{LEM003.4}, condition~\eqref{THM001.2} of Theorem~\ref{THM001} is necessary for the bi-coset graph $\B(G,H,H,HgH)$ to be semisymmetric. Moreover, most almost simple groups $G$ satisfy condition~\eqref{THM001.1}, with all exceptions listing in~\cite[Table~II--VI]{LPS1987}.
\end{remark}

It turns out that conditions~\eqref{THM001.2}--\eqref{THM001.3} in Theorem~\ref{THM001} are not difficult to satisfy for specific cases. We construct two infinite families of biprimitive semisymmetric graphs in Examples~\ref{CON002} and~\ref{CON003} using Theorem~\ref{THM001}, and more examples can be constructed similarly.

There have been some classification results on biprimitive semisymmetric graphs of specific order or valency. The cubic biprimitive semisymmetric graphs are determined by Iofinova and Ivanov~\cite{II1985}. It is known by Du and Maru\v{s}i\v{c}~\cite{DM1999} that the smallest biprimitive semisymmetric graphs have order $80$. In~\cite{DX2000}, Du and Xu classified semisymmetric graphs of order $2pq$, and in particular, the biprimitive ones (see~\cite[Theorem~2.8]{DX2000}). For more recent classification results, we refer to~\cite{CL2024,FW2023,HL2015,LL2022}.

However, no classification results on $G$-biprimitive semisymmetric graphs have been obtained in the general case without specifying the order or valency, even when $G$ is almost simple, the case of all known examples. As the first attempt to classify $G$-biprimitive semisymmetric graphs with almost simple $G$, our second main result classifies those where $G$ is an alternating or symmetric group.

Let $\Gamma$ be a $G$-biprimitive semisymmetric graph. Take vertices $u$ and $w$ from the two parts of $\Gamma$, respectively, and let $L$ and $R$ be the stabilizers in $G$ of $u$ and $w$. A simple derivation (see Subsection~\ref{SUBSEC2.1}) shows that $|L|=|R|$ and that $\Gamma$ is isomorphic to $\B(G,L,R,RgL)$ for some $g\in G$. Clearly, the maximal subgroups $L$ and $R$ in $G$ fall into one of the following cases:
\begin{itemize}
\item $L$ and $R$ are conjugate in $G$;
\item $L$ and $R$ are isomorphic but not conjugate in $G$;
\item $L$ and $R$ have the same order but are not isomorphic.
\end{itemize}
To study the case when $G=\Al_n$ or $\Sy_n$, we are motivated to determine all pairs of maximal subgroups of $\Al_n$ or $\Sy_n$ with the same order, which is addressed below.

\begin{proposition}\label{THM002}
Let $G$ be an almost simple group with socle $\Al_n$, where $n\geq5$, and let $H$ and $K$ be maximal subgroups of $G$ such that $|H|=|K|$. Then, interchanging $H$ and $K$ if necessary, one of the following holds:
\begin{enumerate}[{\rm(a)}]
\item\label{THM002.1} $H$ and $K$ are conjugate in $\Aut(\Al_n)$;
\item\label{THM002.2} $H$ and $K$ are almost simple satisfying one of the following:
\begin{enumerate}[{\rm(b.1)}]
\item\label{THM002.21} $\Soc(H)\cong\Soc(K)$,
\item\label{THM002.22} $\Soc(H)\cong\PSp_{2m}(q)$ and $\Soc(K)\cong\POm_{2m+1}(q)$ with $m\geq3$ and $q$ odd prime power,
\item\label{THM002.23} $\Soc(H)\cong\PSL_3(4)$ and $\Soc(K)\cong\PSL_4(2)$;
\end{enumerate}
\item\label{THM002.4} $H$ and $K$ are primitive of diagonal type on $n$ points such that $\Soc(H)\cong(\PSp_{2m}(q))^k$ and $\Soc(K)\cong(\POm_{2m+1}(q))^k$ with $k\geq2$, $m\geq3$ and $q$ odd prime power;
\item\label{THM002.3} $n=6$, $G=\Al_6$ or $\Sy_6$, and $(H,K)=((\Sy_4\times\Sy_2)\cap G,(\Sy_2\wr\Sy_3)\cap G)$.
\end{enumerate}
\end{proposition}

As a byproduct of proving Proposition~\ref{THM002}, we identify all pairs of almost simple groups of the same order. This result, of independent interest, is presented in the following proposition.

\begin{proposition}\label{THM004}
Let $H$ and $K$ be almost simple groups of the same order such that $\Soc(H)\ncong\Soc(K)$. Then either $\{\Soc(H),\Soc(K)\}=\{\POm_{2m+1}(q),\PSp_{2m}(q)\}$ for some integer $m\geq3$ and odd prime power $q$, or $\{\Soc(H),\Soc(K)\}=\{\PSL_3(4),\PSL_4(2)\}$.
\end{proposition}

Using Proposition~\ref{THM002}, we adopt the analysis and notation in the paragraph above Proposition~\ref{THM002} to classify $G$-biprimitive semisymmetric graphs, where $G$ is an alternating or symmetric group.
This approach yields the second main result of the paper:

\begin{theorem}\label{THM003}
Let $\Gamma$ be a $G$-biprimitive semisymmetric graph such that $G$ is almost simple with alternating socle $T$. Then $\Gamma\cong\B(G,L,R,RgL)$ for some maximal subgroups $L$ and $R$ of $G$ and $g\in G$ such that one of the following holds:
\begin{enumerate}[{\rm(a)}]
\item\label{THM003.1} $L$ and $R$ are conjugate of index $d$ in $G$, and one of the following cases occurs:
\begin{enumerate}
\item[{\rm(a.1)}]\label{THM003.11} there exists an overgroup $H$ of $G$ in $\Sy_d$ whose socle is neither $T$ nor $\Al_d$, and all possible pairs $(G,H)$ with $G$ maximal in $H$ are listed in~\cite[Table~II--VI]{LPS1987};
\item[{\rm(a.2)}]\label{THM003.12} $\Gamma$ is isomorphic to a semisymmetric graph as in Theorem~$\ref{THM001}$;
\item[{\rm(a.3)}]\label{THM003.13} $\Gamma$ is isomorphic to a semisymmetric graph as in Proposition~$\ref{CON001}$.
\end{enumerate}
\item\label{THM003.2} $L$ and $R$ are isomorphic and non-conjugate in $G$, and either $L$ and $R$ are almost simple with the same socle, or $G$ is an alternating group.
\item\label{THM003.3} $L$ and $R$ are of the same order and non-isomorphic, and lie in Proposition~$\ref{THM002}$\eqref{THM002.2}--\eqref{THM002.3}.
\end{enumerate}
\end{theorem}

\begin{remark}\label{REM001}
Let $G$ be almost simple with alternating socle. For a maximal subgroup $L$ of $G$ and $g\in G$, the bi-coset graph $\B(G,L,L,LgL)$ is semisymmetric if $(G,L,g)$ satisfies the conditions of Theorem~\ref{THM001} or Proposition~$\ref{CON001}$. We provide two such semisymmetric graphs in Examples~\ref{CON002} and~\ref{EX1} respectively, which also serve as examples in cases~(\ref{THM003.1}.2) and~(\ref{THM003.1}.3) of Theorem~\ref{THM003}.
\end{remark}

For $G$-biprimitive semisymmetric graphs of twice odd order, where $G=\Al_n$ or $\Sy_n$, we have the following more precise result.

\begin{proposition}\label{CORO003}
Let $G$ be almost simple with socle $\Al_n$ and $n\geq8$, and let $\Gamma$ be a $G$-biprimitive graph with each part of odd size such that $\Gamma$ is neither a complete bipartite graph nor an empty graph. Then $\Gamma\cong\B(G,H,H,HgH)$ for some $g\in G$ and some intransitive or imprimitive subgroup $H$ of $G$ such that one of the following holds:
\begin{enumerate}[{\rm(a)}]
\item\label{CORO003.1} $(G,H,g)$ satisfies Theorem~$\ref{THM001}$\eqref{THM001.3}, and $\Gamma$ is semisymmetric if and only if $HgH\neq Hg^{-1}H$;
\item\label{CORO003.2} $H$ is an imprimitive subgroup of $G=\Al_n$, and $\Gamma$ is semisymmetric if and only if $Hg^{-1}H\neq Hg^{\iota}H$ for every $\iota\in\N_{\Sy_n}(H)\setminus H$.
\end{enumerate}
\end{proposition}

The subsequent sections of this paper unfold as follows. In the next section, we begin with the concept of bi-coset graphs and some preliminary results. After that, we prove Theorem~\ref{THM001} and make use of it to construct two infinity families of biprimitive semisymmetric graphs. Section~\ref{SEC5} is dedicated to proving Proposition~\ref{THM004}, which enables us to establish Proposition~\ref{THM002} in Section~\ref{SEC3}. In the final section, we classify biprimitive semisymmetric graphs of $\Al_n$ and $\Sy_n$ by demonstrating Theorems~\ref{THM003} and Proposition~\ref{CORO003}, and an example of such a graph for Proposition~\ref{CON001} is also presented there.

\section{Proof and application of Theorem~\ref{THM001}}\label{SEC2}

For a graph $\Gamma$, we use $V(\Gamma)$ and $E(\Gamma)$ to denote its vertex set and edge set, respectively. For a group $G$, denote by $\Soc(G)$ the socle of $G$, that is, the product of the minimal normal subgroups of $G$. For a subgroup $H$ of $G$, let $[G\,{:}\,H]$ be the set of right cosets of $H$ in $G$. We use $\mathrm{D}_{2n}$ to denote the dihedral group of order $2n$.

In the following subsection, we introduce the group-theoretic construction of biprimitive graphs and give some fundamental properties of these graphs. We then focus on biprimitive semisymmetric graphs with conjugate stabilizers and establish Theorem~\ref{THM001} in Subsection~\ref{SUBSEC2.2}. In Subsection~\ref{SUBSEC2.3}, we apply Theorem~\ref{THM001} to construct two infinite families of biprimitive semisymmetric graphs.

\subsection{Group-theoretic construction of biprimitive graphs}\label{SUBSEC2.1}

Let $\Gamma$ be a bipartite graph, and let $G$ be a group of automorphisms of $\Gamma$ that preserves each part setwise. Then $\Gamma$ is said to be \emph{$G$-semitransitive} if $G$ acts transitively on both parts. In particular, a $G$-biprimitive graph is $G$-semitransitive. The following concept, as defined in~\cite{DX2000}, gives a group-theoretic construction of semitransitive graphs.

\begin{definition}\label{DEF001}
Let $G$ be a finite group, let $L$ and $R$ be subgroups of $G$, and let $D$ be a union of double cosets of $R$ and $L$ in $G$. A \emph{bi-coset graph} $\B(G,L,R,D)$ of $G$ (with respect to $L,R$ and $D$) is a bipartite graph with parts $[G\,{:}\,L]$ and $[G\,{:}\,R]$ whose edge set is $\{\{Lg,Rdg\}\mid g\in G,d\in D\}$.
\end{definition}

The following lemma gives some elementary properties of bi-coset graphs.

\begin{lemma}[{\cite[Lemmas~2.3~and~2.6]{DX2000}}]\label{LEM003}
Let $\Gamma=\B(G,L,R,D)$ be a bi-coset graph.
\begin{enumerate}[{\rm(a)}]
\item\label{LEM003.1} $\Gamma$ is regular if and only if $|L|=|R|$.
\item\label{LEM003.5} $\Gamma$ is connected if and only if $G=\langle D^{-1}D\rangle$.
\item\label{LEM003.2} $\Gamma$ is isomorphic to $\B(G,L^a,R^b,b^{-1}Da)$
    for each $a,b\in G$.
\item\label{LEM003.3} $G$ acts semitransitively on $\Gamma$ by right multiplication with kernel $\Core_G(L)\cap\Core_G(R)$. Under this action, $G$ is transitive on $E(\Gamma)$ if and only if $D=RgL$ is a single double coset.
\item\label{LEM003.4} If $L=R$ and $D=D^{-1}$, then $\Gamma$ is vertex-transitive.
\end{enumerate}
\end{lemma}

Let $\Gamma$ be a $G$-semitransitive graph with biparts $U$ and $W$. Take any points $u\in U$ and $w\in W$, and set $D=\{g\in G\mid \{u,w^g\}\in E(\Gamma)\}$. It is not difficult to verify (see~\cite[Lemma~2.4]{DX2000} for a proof) that $\Gamma$ is isomorphic to the bi-coset graph $\B(G,G_u,G_w,D)$. According to this and Lemma~\ref{LEM003}\eqref{LEM003.3}, the $G$-semitransitive graphs are in a one-to-one correspondence with bi-coset graphs of $G$. Note also that a bi-coset graph $\B(G,L,R,D)$ is $G$-biprimitive if and only if both $L$ and $R$ are maximal in $G$. If $\Gamma$ is $G$-biprimitive, then the kernel of $G$ on $U$ (or $W$) is transitive on $W$ (or $U$), and so we obtain the next lemma.

\begin{lemma}\label{LEM005}
Let $\Gamma$ be a $G$-biprimitive graph that is neither a complete bipartite graph nor an empty graph. Then $G$ acts faithfully on both of the two parts.
\end{lemma}

We close this subsection with another simple lemma on biprimitive graphs.

\begin{lemma}\label{CORO002}
Let $\Gamma$ be a $G$-biprimitive graph with each part of size $n\geq5$ such that $\Gamma$ is not a complete bipartite graph, a perfect matching, or their bipartite complements. Then $\Soc(G)$ is not isomorphic to $\Al_n$.
\end{lemma}

\begin{proof}
Let $U$ and $W$ be the two parts of $\Gamma$, and take any vertices $u$ and $w$ from $U$ and $W$, respectively. By the assumption and Lemma~\ref{LEM005}, $G$ acts faithfully on both $U$ and $W$. Suppose on the contrary that $G$ has socle $\Al_n$. Since $G$ is primitive on $U$ (and also $W$), it follows that $G$ is isomorphic to $\Al_n$ or $\Sy_n$ with $|G\,{:}\,G_u|=|G\,{:}\,G_w|=n$. If $n=6$ and $\{G_u,G_w\}=\{\Sy_5\cap G,\PGL_2(5)\cap G\}$, then it follows from $G=G_uG_w$ that $G_u$ acts transitively on $W$, contradicting that $\Gamma$ is not a complete bipartite or its bipartite complements. Hence, we derive from~\cite[Theorem~5.2B]{DM1996} that
$G_u=(G_w)^g$ for some $g\in G$, and $G$ acts $2$-transitively on both $U$ and $W$. Let $v=w^g$. Then $G_u$ fixes $u$ and $v$ while acts transitively on $W\setminus\{v\}$. Therefore, the set of vertices adjacent to $u$ is $W$, $\{v\}$, $\emptyset$ or $W\setminus\{v\}$. By the transitivity of $G$ on $U$ and $W$, we conclude that $\Gamma$ is a complete bipartite graph, a perfect matching or their bipartite complements, a contradiction.
\end{proof}

\subsection{Proof of Theorem~\ref{THM001}}\label{SUBSEC2.2}

In this subsection, we prove Theorem~\ref{THM001}, which focuses on the biprimitive semisymmetric graphs with conjugate stabilizers. We first establish three lemmas.

\begin{lemma}\label{LEM001}
Let $\B(G,L,R,RgL)$ be a bi-coset graph with $L$ and $R$ maximal in $G$, let $M$ be an overgroup of $G$, and let $X$ and $Y$ be maximal subgroups of $M$ containing $L$ and $R$, respectively, such that $G(X\cap Y^g)=M$ and neither $X$ nor $Y$ contains $G$. Then the following hold:
\begin{enumerate}[{\rm(a)}]
\item\label{LEM001.1} $\B(G,L,R,RgL)\cong\B(M,X,Y,YgX)$.
\item\label{LEM001.3} Suppose $L=R$ and $X=Y$. Then $RgL=Rg^{-1}L$ if and only if $YgX=Yg^{-1}X$.
\end{enumerate}
\end{lemma}

\begin{proof}
The conclusions are obviously true for $G=M$. So assume $G<M$ in the following. Note that $L\leq G\cap X<G$ and $R\leq G\cap Y<G$. Then since both $L$ and $R$ are maximal in $G$, we conclude that $L=G\cap X$ and $R=G\cap Y$.
We claim that $RgL=(YgX)\cap G$. Clearly, $RgL\subseteq(YgX)\cap G$, and it suffices to show that they have the same size. Let $YgX$ contain exactly $r$ right cosets $Yg_1,Yg_2,\ldots,Yg_r$ of $Y$.
Then
\[
r=\frac{|X|}{|X\cap Y^g|}=\frac{|GX||G\cap X|}{|G||X\cap Y^g|}.
\]
Since $GX\supseteq G(X\cap Y^g)=M$, we have $GX=G(X\cap Y^g)$. Hence, the above equation yields
\begin{equation}\label{EQ001}
r=\frac{|G(X\cap Y^g)||G\cap X|}{|G||X\cap Y^g|}=\frac{|G\cap X|}{|G\cap(X\cap Y^g)|}=\frac{|G\cap X|}{|(G\cap X)\cap (G\cap Y^g)|}=\frac{|L|}{|L\cap R^g|}.
\end{equation}
Since $GY^g\supseteq G(X\cap Y^g)=M$, we have $GY^g=M$, and so $M=YG$. Thus, we may take $g_i\in G$ for each $i\in\{1,\ldots,r\}$. Consequently, it follows from~\eqref{EQ001} that
\[
|(YgX)\cap G|=\left|\bigcup\limits_{i=1}^{r}(Yg_i\cap G)\right|=\left|\bigcup\limits_{i=1}^{r}(Y\cap G)g_i\right|=\left|\bigcup\limits_{i=1}^{r}Rg_i\right|=r|R|=\frac{|R||L|}{|L\cap R^g|}=|RgL|,
\]
proving the claim.

Write $\Gamma=\B(G,L,R,RgL)$ and $\Sigma=\B(M,X,Y,YgX)$.
Define a map $\varphi\colon V(\Gamma)\to V(\Sigma)$ by letting $\varphi$ send $La$ to $Xa$ and $Ra$ to $Ya$ for all $a\in G$. Making use of $L=G\cap X$, $R=G\cap Y$ and $M=XG=YG$, one may directly verify that $\varphi$ is well-defined and bijective. Moreover, for $a,b\in G$, the above claim implies that
\[
\{La,Rb\}\in E(\Gamma) \Leftrightarrow
ba^{-1}\in RgL=(YgX)\cap G \Leftrightarrow
ba^{-1}\in YgX \Leftrightarrow
\{Xa,Yb\}\in E(\Sigma).
\]
Therefore, $\varphi$ is a graph isomorphism from $\Gamma$ to $\Sigma$, which confirms statement~\eqref{LEM001.1}.

Now suppose $L=R$ and $X=Y$. If $RgL=Rg^{-1}L$, then
\[
YgX=Y(RgL)X=Y(Rg^{-1}L)X=Yg^{-1}X.
\]
Conversely, if $YgX=Yg^{-1}X$, then we conclude from $RgL=(YgX)\cap G$ that
\[
RgL=(YgX)\cap G=(Yg^{-1}X)\cap G=((YgX)\cap G)^{-1}=(RgL)^{-1}=Rg^{-1}L.
\]
This proves statement~\eqref{LEM001.3}.
\end{proof}

\begin{lemma}\label{LEM004}
Let $\Gamma=\B(G,H,H,D)$ be a bi-coset graph with $H$ maximal, nontrivial and core-free in $G$, and let $R$ be the right multiplication action of $G$ on $V(\Gamma)$. If there exists an automorphism of $\Gamma$ that interchanges the two parts of $\Gamma$ and centralizes $R(G)$, then $D=D^{-1}$.
\end{lemma}

\begin{proof}
Let $\iota$ be the automorphism of $\Gamma$ described in the lemma, and let $U$ and $W$ be the biparts of $\Gamma$. To distinguish the vertices of two parts, we write the vertices from $U$ as $Hg$ while the vertices from $W$ as $Kg$, where $K=H$ and $g\in G$. Take an arbitrary $a\in G$, and let $(Ha)^\iota=Kb$. As $\iota$ centralizes $R(G)$, we deduce that for each $h\in H$,
\[
Kb=(Ha)^\iota=(Ha)^{R(a^{-1}ha)\iota}=(Ha)^{\iota R(a^{-1}ha)}=(Kb)^{R(a^{-1}ha)}=Kba^{-1}ha.
\]
This implies that $(ba^{-1})h(ab^{-1})\in K=H$, and so $ab^{-1}\in\N_G(H)$. Since $H$ is maximal and core-free in $G$, it follows that $\N_G(H)=H$ and hence $(Ha)^\iota=Kb=Ka$. With a similar argument we have $(Ka)^\iota=Ha$ for each $a\in G$. Considering the action of $\iota$, we conclude that
\[
a\in D\Leftrightarrow \{H,Ka\}\in E(\Gamma)\Leftrightarrow \{Ha,K\}\in E(\Gamma)\Leftrightarrow a^{-1}\in D,
\]
which completes the proof.
\end{proof}

\begin{lemma}\label{LEM002}
Let $M$ be an almost simple group with socle $T$, and let $A=M.2$. Then $A$ is either almost simple with socle $T$ or isomorphic to $M\times2$.
\end{lemma}

\begin{proof}
Note that $T$ is normal in $A$, as $T$ is characteristic in $M$ while $M$ is normal in $A$. Since $\C_A(T)\cap G=\C_M(T)=1$, we have $\C_A(T)=1$ or $2$. If $\C_A(T)=1$, then
\[
T\leq A\cong A/\C_A(T)\lesssim\Aut(T),
\]
and so $A$ is almost simple with socle $T$. If $\C_A(T)=2$, then $A=M\times\C_A(T)=M\times2$.
\end{proof}

We are now ready to prove Theorem~\ref{THM001}.

\begin{proof}[Proof of Theorem~$\ref{THM001}$]
Let $\Gamma=\B(G,H,H,HgH)$ with biparts $U$ and $W$, let $A=\Aut(\Gamma)$, and let $A^+$ be the subgroup of $A$ fixing setwise $U$ and $W$. Since $H$ is maximal in $G$, the graph $\Gamma$ is $G$-biprimitive and so $(A^+)$-biprimitive. Since $HgH\neq Hg^{-1}H$, the graph $\Gamma$ is not a complete bipartite graph, a perfect matching, or their bipartite complements. Then we derive from Lemmas~\ref{LEM005} and~\ref{CORO002} that $A^+$ acts faithfully on $U$ and $W$, respectively, and that $\Soc(A^+)\neq\Al_n$. Write $M=\Aut(T)$. Since $A^+$ is an overgroup of $G$, it follows from condition~\eqref{THM001.1} that $\Soc(A^{+})=T$ and hence $A^+\leq M$.

Take a core-free maximal subgroup $K$ of $M$ as in~\eqref{THM001.3}. Then $K\geq H$ and $M=G(K\cap K^{g})$, that is, the conditions of Lemma~\ref{LEM001} hold with $L=R=H$ and $X=Y=K$. Thus, $\B(G,H,H,HxH)\cong\B(M,K,K,KgK)$, and in particular, $M\leq A^{+}$.

So far we have shown $M=A^{+}$. Suppose for a contradiction that $\Gamma$ is vertex-transitive, namely, $A=A^{+}.2$. Then by Lemma~\ref{LEM002}, either $A$ is almost simple with socle $T$, or $A=M\times2$. The former case is not possible as it would imply $A\leq\Aut(T)=M=A^+$. Therefore, $A=\Aut(T)\times2=M\times2$, and so Lemma~\ref{LEM004} implies $KgK=Kg^{-1}K$. We then conclude by Lemma~\ref{LEM001} that $HgH=Hg^{-1}H$, contradicting condition~\eqref{THM001.2}. This completes the proof.
\end{proof}

\begin{remark}\label{RMK002}
In the above proof of Theorem~\ref{THM001}, it can be observed that the theorem still holds if condition~\eqref{THM001.1} is replaced by ``$\Soc(A^+)=T$ or $\Al_n$'', where $A^+$ is the automorphism group of $\B(G,H,H,HgH)$ that preserves each part setwise. This observation will be used in the proof of Proposition~\ref{CORO003} to address biprimitive graphs of order $990$.
\end{remark}

\subsection{Infinite families of biprimitive semisymmetric graphs}\label{SUBSEC2.3}

Applying Theorem~\ref{THM001}, one may construct many infinite families of biprimitive semisymmetric graphs. In this subsection, we present two infinite families of such graphs admitting alternating and symplectic groups.

Let $v=\{V_1,\ldots,V_k\}$ and $w=\{W_1,\ldots,W_k\}$ be two partitions of $\{1,\ldots,n\}$, where each block $V_i$ and $W_j$ has the same size. A matrix $P=(|V_i\cap W_j|)_{k\times k}$ is called an \emph{intersection matrix} of $v$ and $w$. Note that $P$ might be different by changing the orderings within $\{V_1,\ldots,V_k\}$ and $\{W_1,\ldots,W_k\}$, but all the possibilities belong to the set
\[
\{X(|V_i\cap W_j|)_{k\times k}Y\mid\text{$X$ and $Y$ are $k\times k$ permutation matrices}\}.
\]
The following lemma characterizes suborbits of the imprimitive subgroups of symmetric groups via intersection matrices of partitions.

\begin{lemma}[{\cite[Lemma~2.9(a)]{YFX2023}}]\label{LEM012}
Let $v$ be a partition of $\{1,\ldots,n\}$, let $K$ be the stabilizer of $v$ in $\Sy_n$, let $g\in\Sy_n$, and let $P$ be an intersection matrix of $v$ and $v^g$. Then an element $h$ of $\Sy_n$ is in $KgK$ if and only if there exist $k\times k$ permutation matrices $X$ and $Y$ such that $XPY$ is an intersection matrix of $v$ and $v^h$.
\end{lemma}

We now give our first family of biprimitive graphs and show their semisymmetry using Theorem~\ref{THM001}.

\begin{example}\label{CON002}
Let $n=3m$ for some integer $m\geq8$, and let $v=\{V_1,V_2,V_3\}$, where $V_{i+1}=\{im+1,\cdots,im+m\}$ for $i\in\{0,1,2\}$, be a partition of $\{1,\ldots,n\}$. Let $G=\Al_n$ be the alternating group on $\{1,\ldots,n\}$, let $H$ be the stabilizer of the partition $v$ in $G$, and let
\[
g=(1,m+1)(2,m+2,2m+1)(3,2m+2,4,2m+3)\in G.
\]
Then the bi-coset graph $\B(G,H,H,HgH)$ is $G$-biprimitive semisymmetric.
\end{example}

\begin{proof}[Proof of Example~$\ref{CON002}$]
We first prove $HgH\neq Hg^{-1}H$. Let $W_i=(V_i)^g$ for $i\in\{1,2,3\}$, let $w$ be the partition $\{W_1,W_2,W_3\}$ of $\{1,\ldots,n\}$, and let $P=(|V_i\cap W_j|)_{3\times3}$. It is straightforward to verify that
\[
P=
\begin{pmatrix}
m-4 & 1 & 3 \\
2 & m-2 & 0 \\
2 & 1 & m-3
\end{pmatrix}.
\]
Let $U_i=(V_i)^{g^{-1}}$ for $i\in\{1,2,3\}$, and let $Q=(|V_i\cap U_j|)_{3\times3}$. Then
\[
Q=P^\mathsf{T}=
\begin{pmatrix}
m-4 & 2 & 2 \\
1 & m-2 & 1 \\
3 & 0 & m-3
\end{pmatrix}.
\]
Observe in $P$ and $Q$ that none of the three diagonal entries $m-4$, $m-2$ and $m-3$ equals any of the non-diagonal entries. Thus, there are no permutation matrices $X$ and $Y$ such that $Q=XPY$. It follows from Lemma~\ref{LEM012} that $g^{-1}\notin HgH$ and hence $HgH\neq Hg^{-1}H$, as desired.

By examining the possible overgroups in~\cite[Tables~II--VI]{LPS1987} of the primitive group $G$ on $[G\,{:}\,H]$, we deduce that, except for the alternating and symmetric groups on $[G\,{:}\,H]$, each overgroup of $G$ has socle $\Al_n$.

So far we have verified that the triple $(G,H,g)$ satisfies conditions~\eqref{THM001.2} and~\eqref{THM001.1} of Theorem~\ref{THM001}. Now let $K$ be the stabilizer of the partition $v$ in $\Sy_n$. Then $H\leq K$, and the transposition $(m+4,m+5)$ is in $K\cap K^g$. Thus, $G(K\cap K^g)=\Sy_n$, and Theorem~\ref{THM001}\eqref{THM001.3} also holds for $(G,H,g)$. Moreover, since $H$ is maximal in $G$, the bi-coset graph $\B(G,H,H,HgH)$ is $G$-biprimitive. We then conclude by Theorem~\ref{THM001} that $\B(G,H,H,HgH)$ is semisymmetric.
\end{proof}

The second family is constructed based on the imprimitive subgroups of symplectic groups, and its construction captures the same essence as the graphs in Example~\ref{CON002}.

\begin{example}\label{CON003}
Let $V$ be a non-degenerate symplectic space of dimension $2mk$ over $\mathbb{F}_q$, where $m\geq8$, $k\geq3$ and $q$ is a prime power. Let $V=V_1\oplus\cdots\oplus V_k$ be a non-degenerate decomposition such that each $V_i$ has dimension $2m$, and let $\mathcal{D}=\{V_1\ldots,V_k\}$. For each $i\in\{1,\ldots,k\}$, take a symplectic basis $e_{i1},f_{i1},e_{i2},f_{i2},\ldots,e_{im},f_{im}$ of $V_i$. Suppose that $G=\PSp(V)$ and that $H$ is the stabilizer of $\mathcal{D}$ in $G$. Take $g\in G$ to be the image of
\[
a=(e_{11},e_{21})(f_{11},f_{21})(e_{12},e_{22},e_{31})(f_{12},f_{22},f_{31})
(e_{13},e_{32},e_{14},e_{33})(f_{13},f_{32},f_{14},f_{33})
\]
modulo scalars. Then the $G$-biprimitive bi-coset graph $\B(G,H,H,HgH)$ is semisymmetric.
\end{example}

\begin{proof}[Proof of Example~$\ref{CON003}$]
We first prove $HgH\neq Hg^{-1}H$. For a block matrix $B=(B_{ij})_{k\times k}$ with each $B_{ij}$ being a $2m\times2m$ matrix, define its rank matrix $R(B)=(\mathrm{rank}(B_{ij}))_{k\times k}$. Fix the basis $e_{11},f_{11},\ldots,e_{1m},f_{1m},\ldots\ldots,e_{k1},f_{k1},\ldots,e_{km},f_{km}$
for $V$, and view the elements of $\GSp(V)$ as the matrices with respect to this basis. Then direct computation shows that
\[
R(a)=
\begin{pmatrix}
2m-8 & 2 & 6 &  \\
4 & 2m-4 &  &  \\
4 & 2 & 2m-6 &  \\
& & & 2mI_{k-3}
\end{pmatrix}
\]
and $R(a^{-1})=R(a)^\mathsf{T}$. Let $N$ be the stabilizer of $\mathcal{D}$ in $\GSp(V)$, the conformal symplectic group. Then for each element $b\in NaN$, there exist permutation matrices $X$ and $Y$ such that $R(b)=XR(a)Y$. Moreover, let $Z$ be the center of $\GSp(V)$, and observe that $R(bz)=R(b)$ for each $b\in NaN$ and $z\in Z$. Consequently, the rank matrices of the elements in $NaNZ$ have the form $XR(a)Y$. Since none of the diagonal entries $2m-8$, $2m-4$, $2m-6$ and $2m$ equals any of the non-diagonal entries in $R(a)$, there are no permutation matrices $X$ and $Y$ such that $R(a)^\mathsf{T}=XR(a)Y$. Then since $R(a^{-1})=R(a)^\mathsf{T}$, we derive that $a^{-1}\notin NaNZ$ and so $g^{-1}\notin HgH$, as desired.

Examining the possible overgroups in~\cite[Tables~II--VI]{LPS1987} of the primitive group $G$ on $[G\,{:}\,H]$, we deduce that, except for the alternating and symmetric groups on $[G\,{:}\,H]$, each overgroup of $G$ has socle $G$. So the triple $(G,H,g)$ satisfies conditions~\eqref{THM001.2} and~\eqref{THM001.1} of Theorem~\ref{THM001}. Take a generator $\delta$ of the multiplicative group of $\mathbb{F}_q$. Let $\lambda$ be the similarity which maps each $e_{ij}$ to $\delta e_{ij}$ while fixing each $f_{ij}$, and let $\phi\colon V\to V$ be the mapping such that
\[
\Big(\sum_{i,j}x_{ij}e_{ij}+y_{ij}f_{ij}\Big)^\phi
=\sum_{i,j}x_{ij}^pe_{ij}+y_{ij}^pf_{ij}.
\]
Then $\langle\Sp(V),\lambda\rangle=\GSp(V)$, $\langle\GSp(V),\phi\rangle=\CGammaSp(V)$, and both $\lambda$ and $\phi$ stabilize $\mathcal{D}$ and $\mathcal{D}^a$.
Hence, the stabilizer $K$ modulo scalars of $\mathcal{D}$ in $\PCGammaSp(V)$ satisfies that $K\cap K^g\nleq G$ and so $G(K\cap K^g)=\PCGammaSp(V)=\Aut(G)$. Thus, Theorem~\ref{THM001}\eqref{THM001.3} also holds for $(G,H,g)$, and we conclude that $\B(G,H,H,HgH)$ is a $G$-biprimitive semisymmetric graph.
\end{proof}

\section{Almost simple groups of the same order}\label{SEC5}

The purpose of this section is to establish Proposition~\ref{THM004}, which will be used in the proof of Proposition~\ref{THM002} to address maximal almost simple subgroups of alternating groups. The primary ideas for the proof are presented by Artin~\cite{Artin1955} in 1955, where most cases of simple groups of the same order are considered. Based on Artin's method, Kimmerle, Lyins, Sandling and Teague~\cite{KLST1990} examined all simple groups and demonstrated that $\{\POm_{2m+1}(q),\PSp_{2m}(q)\}$ and $\{\PSL_4(2),\PSL_3(4)\}$ are the only pairs of simple groups of the same order. Following the framework outlined in~\cite{KLST1990}, we investigate the dominant prime and Artin invariants in first two subsections and complete the proof of Proposition~\ref{THM004} in Subsection~\ref{subsec5.3}.

\subsection{Groups of Lie type whose characteristic contributes the most to the order}\label{SUBSEC5.1}

Let $n$ be a positive integer. For a prime $p$, denote by $n_p$ the largest power of $p$ dividing $n$. If $n\geq2$, then the \emph{dominant prime} $r$ of $n$ is the unique prime such that $n_r>n_p$ for any prime $p\neq r$.
The aim of this subsection is to establish Propositions~\ref{PROP011} and~\ref{PROP012}, which state that except for a few cases, the order of every almost simple group of Lie type has the dominant prime equal to its characteristic.

Let $q\geq2$ and $m\geq2$ be integers. A prime number is called a \emph{primitive prime divisor} of the pair $(q,m)$ if it divides $q^m-1$ but does not divide $q^i-1$ for any positive integer $i<m$. By Fermat's Little Theorem, each primitive prime divisor $r$ of $(q,m)$ satisfies that $r\equiv1\ (\bmod\ m)$ and so $r>m$. The existence of primitive prime divisors is revealed in the following Zsigmondy's Theorem (see~\cite[Theorem~IX.8.3]{HB1982} for example), where the final statement follows from Fermat's Little Theorem.

\begin{lemma}[{Zsigmondy's Theorem}]\label{lem0.10}
There exists a primitive prime divisor for each pair $(q,m)$, except when $(q,m)=(2,6)$ or $(q,m)=(2^f-1,2)$ for some integer $f>1$.
\end{lemma}

For an integer $n$, denote $Q(n)=n_r$, where $r$ is the dominant prime to $n$.

\begin{lemma}[{\cite[Page~464]{Artin1955}}]\label{LEM014}
Let $n$ and $q$ be positive integers, and let
\[
f_1(n,q)=Q\left(\prod_{i=1}^{n}(q^i-1)\right),\ \
f_2(n,q)=Q\left(\prod_{i=1}^{n}(q^i-(-1)^i)\right),\ \
f_3(n,q)=Q\left(\prod_{i=1}^{n}(q^{2i}-1)\right).
\]
\begin{enumerate}[{\rm(a)}]
\item\label{LEM014.1} If $q$ is even, then $f_i(n,q)\leq3^{\frac{n}{2}}(q+1)^n$ for each $i\in\{1,2,3\}$.
\item\label{LEM014.2} If $q$ is odd, then $f_i(n,q)\leq2^n(q+1)^n$ for $i\in\{1,2\}$, and $f_3(n,q)\leq4^n(q+1)^n$.
\end{enumerate}
\end{lemma}

We are now ready to determine the dominant prime for classical almost simple groups.

\begin{proposition}\label{PROP011}
Let $G$ be a classical almost simple group over a field $\mathbb{F}$. Then the dominant prime $r$ of $|G|$ equals the characteristic of $\mathbb{F}$ except for the following cases:
\begin{enumerate}[{\rm(a)}]
\item\label{PROP011.1} $G$ has socle $\PSL_2(p)$ with $p$ a Mersenne prime or a Fermat prime, and $r=2$;
\item\label{PROP011.2} $G$ has socle $\PSL_2(2^f)$ with $r=2^f+1$ a Fermat prime;
\item\label{PROP011.3} $r=2$, and $G=\PSL_2(9).2$, $\PGammaL_2(9)$, $\PGammaL_2(25)$, $\PGammaL_2(49)$, $\PGammaL_2(81)$, $\PSL_3(3).2$, $\PGSp_4(3)$, $\PSU_3(3)$, $\PSU_3(3).2$ or $\PSU_4(3).\mathrm{D}_8$;
\item\label{PROP011.4} $r=3$, and $G=\PSL_2(8)$, $\PGammaL_2(8)$, $\PGU_3(8).3$ or $\PSU_4(2)$.
\end{enumerate}
\end{proposition}

\begin{proof}
Let $\mathbb{F}$ have characteristic $p$ and order $q=p^f$, and let $T$ be the socle of $G$. Suppose that the dominant prime $r$ of $|G|$ is not equal to $p$.
Take $d$ to be the integer given in~\cite[Table~5.1.A]{KL1990} (it is usually the order of the multiplier of $T$). Then $d$ divides $|\Out(T)|$, and so
\begin{equation}\label{eq009}
|T|_p\leq|G|_p<|G|_r\leq\left(\frac{|\Out(T)|}{d}\right)_r\cdot(d|T|)_r\leq\frac{|\Out(T)|}{d}\cdot(d|T|)_r.
\end{equation}
Moreover, checking the orders of classical simple groups, we see that
\begin{align*}
\begin{aligned}
&(d|\PSL_n(q)|)_r\leq f_1(n,q)\text{ for }n\geq2,
\ \ (d|\PSp_{2n}(q)|)_r\leq f_3(n,q)\text{ for }n\geq2,\\
&(d|\PSU_n(q)|)_r\leq f_2(n,q)\text{ for }n\geq3,
\ \ (d|\POm_{2n+1}(q)|)_r\leq f_3(n,q)\text{ for }n\geq3,\\
&(d|\POm_{2n}^+(q)|)_r\leq f_3(n,q)\text{ for }n\geq4,
\ \ (d|\POm_{2n}^-(q)|)_r\leq f_3(n,q)\text{ for }n\geq4,
\end{aligned}
\end{align*}
where $f_1$, $f_2$ and $f_3$ are the functions in Lemma~\ref{LEM014}.
Combining this with Lemma~\ref{LEM014} and substituting the respective upper bound into~\eqref{eq009}, we obtain an inequality for $|T|_p$, which leads to one of the following:
\begin{enumerate}[{\rm(i)}]
\item\label{PROP011.C1} $T=\PSL_2(q)$, $\PSL_3(q)$ or $\PSU_3(q)$;
\item\label{PROP011.C2} $n=2$, and $T=\PSp_4(3)$ or $\PSp_4(4)$;
\item\label{PROP011.C3} $n=4$, and $q\leq9$;
\item\label{PROP011.C4} $n=5$, and $q\leq3$;
\item\label{PROP011.C5} $n=6$, and $T=\PSL_6(2)$ or $\PSU_6(2)$.
\end{enumerate}
Direct computation shows that, up to isomorphism, $\PGSp_4(3)$, $\PSU_4(2)$ and $\PSU_4(3).\mathrm{D}_8$ are the only three groups of $G$ in~\eqref{PROP011.C2}--\eqref{PROP011.C5} such that $r\neq p$, and the corresponding pair $(r,G)$ lies in~\eqref{PROP011.3} or~\eqref{PROP011.4} of the Proposition~\ref{PROP011}.

For~\eqref{PROP011.C1}, we first deal with $\PSL_3(q)$ and $\PSU_3(q)$. For convenience, we let $T=\PSL_3(q)$ with $q$ possibly negative to cover the case $T=\PSU_3(q)$. Then~\eqref{eq009} leads to
\[
|q|^3=|T|_p
\leq2f\cdot|q^3(q^3-1)(q^2-1)|_r
=2f\cdot|(q^2+q+1)(q+1)(q-1)^2|_r.
\]
Since $r\neq p$, we have $\gcd(q^2+q+1,q+1)_r=1$, $\gcd(q^2+q+1,(q-1)^2)_r\leq(3|q|)_r\leq3$ and $\gcd(q+1,(q-1)^2)_r\leq4_r\leq4$, whence $|(q^2+q+1)(q+1)(q-1)^2|_r\leq4(q-1)^2$. This yields
\[
|q|^3\leq2f\cdot|(q^2+q+1)(q+1)(q-1)^2|_r\leq8f(q-1)^2,
\]
which implies $|p|\leq7$ and $|q|\leq2^5$.
Then a straightforward calculation verifies that either $G\in\{\PSL_3(3).2,\PSU_3(3),\PSU_3(3).2\}$ and $r=2$ as in~\eqref{PROP011.3}, or $(r,G)=(3,\PGU_3(8).3)$ as in~\eqref{PROP011.4}.

For the rest of the proof, let $T=\PSL_2(q)$ (with $q>3$ since $T$ is simple). Then
\[
|G|\,\text{ divides }\,fq(q+1)(q-1).
\]
If $f=1$, then $p<|G|_r\leq(p+1)_r(p-1)_r$, and so $r=2$ is the only prime divisor of $p+1$ or $p-1$, which means that $q=p$ is a Mersenne prime or a Fermat prime, as in~\eqref{PROP011.1}.

Next assume that $f\geq2$ and $r=2$. Then by Lemma~\ref{lem0.10} and $p\neq r$, there exists a primitive prime divisor of $(p,2f)$ greater than $2f$. Hence, $(2f(q+1))_2<q$, and then the fact
\[
q<|G|_2\leq(f(q+1)(q-1))_2\leq\max\{(2f(q+1))_2,(2f(q-1))_2\}
\]
yields that $q<(2f(q-1))_2$. If $f$ is odd, then this inequality indicates $q<(2(q-1))_2$, which is not possible since there exists a primitive prime divisor of $(p,f)$. Thus, $f$ is even, and so
\[
q<(2f(q-1))_2=\left(2f\big(\sqrt{q}+1\big)\big(\sqrt{q}-1\big)\right)_2
\leq4f(\sqrt{q}+1),
\]
which holds only when $q\in\{3^2,5^2,7^2,3^4\}$. Checking these possibilities, we conclude that $G\in\{\PSL_2(9).2,\PGammaL_2(9),\PGammaL_2(25),\PGammaL_2(49),\PGammaL_2(81)\}$, as in~\eqref{PROP011.3}.

Now assume that $f\geq2$ and $r>2$. Since $|G|_p\geq q>f$, the dominant prime $r$ of $|G|$ must divide $q+1$ or $q-1$. If $r$ divides $q-1$, then $q<|G|_r\leq(f(q-1))_r$ and in particular, neither $q=64$ nor $f=2$, which implies that there is a primitive prime divisor of $(p,f)$, contradicting $q<(f(q-1))_r$. Therefore, $r$ divides $q+1$. If there is no primitive prime divisor of $(p,2f)$, then Lemma~\ref{lem0.10} implies that $q=8$, as in~\eqref{PROP011.4}. Now let $s$ be a primitive prime divisor of $(p,2f)$. Then $s$ divides $p^f+1=q+1$, and $s>2f$. Thus, the inequality
\[
q<|G|_r\leq(f(q+1))_r=\left(\frac{f(q+1)}{s}\right)_r(s)_r
\]
implies that $s=r$ and $|G|_r=q+1=|T|_r$. Consequently, $r$ is also the dominant prime of $|T|$. According to~\cite[Theorem~1]{Artin1955}, $T=\PSL_2(2^f)$ with $r=2^f+1$ a Fermat prime, as in~\eqref{PROP011.2}. This completes the proof.
\end{proof}

For a positive integer $m$ and prime $r$, let $\Phi_m$ be the $m$th cyclotomic polynomial, and let $\ord_r(m)$ denote the \emph{order of $m$ modulo $r$}, namely, the smallest positive integer $k$ such that $m^k\equiv1\ (\bmod\ r)$. The following lemma (see~\cite[Lemma~IX.8.1]{HB1982} for a proof) identifies which $\Phi_m(q)$ is divisible by a given prime factor $r$ not dividing $q$.

\begin{lemma}\label{LEM006}
Let $q$ be an integer, let $r$ be a prime not dividing $q$, and let $k=\ord_r(q)$ be the order of $q$ modulo $r$. Then $r$ divides $\Phi_m(q)$ if and only if $m=kr^i$ for some integer $i\geq0$.
\end{lemma}

The proposition below determines the dominant prime for almost simple groups of exceptional Lie type.

\begin{proposition}\label{PROP012}
Let $G$ be an almost simple exceptional group of Lie type over a field $\mathbb{F}$. Then the dominant prime $r$ of $|G|$ coincides with the characteristic of $\mathbb{F}$.
\end{proposition}

\begin{proof}
Let $\mathbb{F}$ have characteristic $p$ and order $q=p^f$, and let $T$ be the socle of $G$. Suppose for a contradiction that $r\neq p$.

First assume that $T={^{2}\B}_2(q)$, where $q=2^f$ with odd $f\geq3$. Then $|G|$ divides $fq^2(q^2+1)(q-1)$, and clearly, $r$ must be a divisor of $q^2+1$. By  Lemma~\ref{lem0.10}, there exists a primitive prime divisor $s$ of $(p,4f)$ such that $s>4f$. Hence, the inequality
\[
q^2<|G|_r\leq(f(q^2+1))_r=\left(\frac{f(q^2+1)}{s}\right)_r(s)_r
\]
holds only when $s=r$ and $(q^2+1)_r=q^2+1$. This implies that $q^2+1$ is prime. However, $q^2+1=4^f+1$ is divisible by $5$ as $f$ is odd, a contradiction.

Now assume $T\neq{^{2}\B}_2(q)$. Direct computation shows that $T$ cannot be $^3\D_4(2)$, $^3\D_4(4)$ or $^3\D_4(8)$. Therefore, $|\Out(T)|/d\leq q$, where $d$ is the integer given in~\cite[Table~5.1.B]{KL1990}. For the case when $G=T$, the proposition is demonstrated in~\cite[Theorem~3.3]{KLST1990} by estimating cyclotomic polynomials $\Phi_m(q)$ in the cyclotomic factorization of $|T|$. In the following, we apply this estimation for $(d|T|)_r$. Consider the cyclotomic factorization in terms of $q$:
\[
d|T|=q^h\prod_{m}\Phi_m(q)^{\gamma_m}.
\]
Let $\delta_m$ be the smallest integer such that $\Phi_m(q)<q^{\delta_m}$ for all $q\geq2$. Then
\[
(d|T|)_r\leq\prod_{\substack{m\\r\mid\Phi_m(q)}}\Phi_m(q)^{\gamma_m}<
\prod_{\substack{m\\r\mid\Phi_m(q)}}q^{\delta_m\gamma_m}.
\]
Applying Lemma~\ref{LEM006}, we derive that $(d|T|)_r$ is bounded from above by $q^{\sum\delta_m\gamma_m}$, where the sum is taken over all $m$ of the form $m=kr^i$ with $i\geq0$. Note from the list of simple groups of exceptional Lie type that $m\leq30$. All the possible values of $\sum\delta_m\gamma_m$ are listed in~\cite[Table~C.3]{KLST1990}, and more precisely, one of the following occurs in~\cite[Table~C.3]{KLST1990}:
\begin{itemize}
\item $r\leq7$, $k=1$, and $\sum\delta_m\gamma_m$ is given in the rows labelled $\sum r^i$;
\item $r\leq7$, $k\neq1$, and $\sum\delta_m\gamma_m$ is given in the rows labelled $\times k$;
\item $r>7$ divides $\Phi_m(q)$ for at most one $m$, and $\sum\delta_m\gamma_m$ is given in the remaining rows.
\end{itemize}
Checking these rows, we conclude that either $\sum\delta_m\gamma_m\leq h-1$, or $r=2$ and $T\in\{\G_2(q),{^{2}\G}_2(q)\}$. For the latter case, we have
\[
|T|_r=|^{2}\G_2(q)|_2=\big((q-1)(q^3+1)\big)_2=\big(q^2-1\big)_2<q^2=q^{h-1}
\]
or
\[
|T|_r=|\G_2(q)|_2=\big((q^2-1)(q^6-1)\big)_2=\big((q^2-1)^2\big)_2<q^4<q^{h-1}.
\]
Hence,
\[
q^h<|G|_r\leq\left(\frac{|\Out(T)|}{d}\right)_r\cdot(d|T|)_r\leq q\cdot q^{h-1}=q^h,
\]
which is a contradiction. This finishes the proof.
\end{proof}

\subsection{Logarithmic proportion and Artin invariants}\label{subsec5.2}

The Logarithmic proportion and Artin invariants are useful quantities in distinguishing the orders of simple groups~\cite{Artin1955}. It turns out that they also apply to the orders of almost simple groups. In this subsection, we aim to determine the intervals of Logarithmic proportion and Artin invariants for almost simple groups with a given socle.

Let $n$ be an integer greater than $1$, and recall that $Q(n)=n_r$, where $r$ is the dominant prime to $n$. Define the \emph{logarithmic proportion} $\lambda(n)$ of $n$ as $\log Q(n)/\log n$. For a group $G$, we simply write $Q(G)=Q(|G|)$ and $\lambda(G)=\lambda(|G|)$. The following lemma gives an upper bound for the logarithmic proportion of alternating and symmetric groups.

\begin{lemma}\label{LEM010}
Let $G$ be an almost simple group with socle $\Al_n$. If $n>20$, then $\lambda(G)<0.3$.
\end{lemma}

\begin{proof}
The upper bounds of $\lambda(\Al_n)$ for any $n$ are listed in~\cite[Table~L.4]{KLST1990}. In particular, $\lambda(\Al_n)\leq0.27$ when $n>20$. Suppose that $G=\Sy_n$. Since $|\Al_n|>20!/2>2^{50}$, we obtain
\[
\lambda(G)=\frac{\log Q(G)}{\log|G|}<\frac{\log2+\log Q(\Al_n)}{\log|\Al_n|}=\frac{\log2}{\log|\Al_n|}+\lambda(\Al_n)<0.02+0.27<0.3,
\]
as the lemma assets.
\end{proof}

The logarithmic proportion of groups of Lie type is presented in Lemma~\ref{LEM007} with the help of the following well-known fact (see, for example,~\cite{Stefan}).

\begin{lemma}\label{LEM008}
For every finite nonabelian simple group $T$ we have $|\Out(T)|<\log_2|T|$.
\end{lemma}

\begin{lemma}\label{LEM007}
Let $G$ be an almost simple group of Lie type with socle $T$. Suppose that $|T|>2^{54}$. Then either $T\neq\PSL_2(q)$ and $\lambda(G)>0.3375$, or $T=\PSL_2(q)$ and $\lambda(G)>0.3$.
\end{lemma}

\begin{proof}
As $|T|>2^{54}$, we derive from Lemma~\ref{LEM008} that $|\Out(T)|<\log_2|T|<|T|^{1/9}$. Hence,
\[
\lambda(G)=\frac{\log Q(G)}{\log|G|}\geq\frac{\log Q(T)}{\log(|T|\cdot|\Out(T)|)}>\frac{9\log Q(T)}{10\log|T|}=\frac{9}{10}\lambda(T).
\]
Moreover, according to~\cite[Proposition~3.5]{KLST1990} (see also~\cite[Table~L.2]{KLST1990}), it holds $\lambda(T)\geq3/8$ when $T\neq\PSL_2(q)$, and $\lambda(T)\geq1/3$ when $T=\PSL_2(q)$. Then the lemma follows immediately by substituting these into the above inequality.
\end{proof}

For the rest of this subsection, we focus on the Artin invariants of groups of Lie type. The following definition follows~\cite{Artin1955} and~\cite{KLST1990}.

\begin{definition}\label{DEF002}
Let $G$ be a finite group, and let $r=r(G)$ be the dominant prime of $|G|$. Write
\begin{align*}
\ell(G)&=\log_rQ(G),\\
\omega(G)&=\max\{\ord_s(r)\mid s
\text{ is a prime divisor of } |G|/Q(G)\},\\
\psi(G)&=\max\big(\{\ord_s(r)\mid s
\text{ is a prime divisor of } |G|/Q(G)\}\setminus\{\omega(G)\}\cup\{0\}\big).
\end{align*}
We refer to $r(G)$, $\ell(G)$, $\omega(G)$ and $\psi(G)$ as the \emph{Artin invariants} of $G$. Moreover, define
\[
F_1(G)=\frac{\ell(G)}{\omega(G)}\ \ \text{ and }\ \ F_2(G)=\frac{\omega(G)}{\omega(G)-\psi(G)}-2F_1(G).
\]
\end{definition}


Recall that the Artin invariant $r(G)$ for almost simple groups $G$ of Lie type is determined in Propositions~\ref{PROP011} and~\ref{PROP012}. In the final proposition of this subsection, we determine the other three Artin invariants of $G$ and estimate the values of $F_1(G)$ and $F_2(G)$.

\begin{proposition}\label{LEM013}
Let $G$ be an almost simple group of Lie type over a field of order $q=p^f$, where $p$ is a prime and $f$ is a positive integer, and let $T$ be the socle of $G$. Suppose
\begin{equation}\label{eq0.6}
|T|=\frac{1}{d}p^{h}\prod_{m}\Phi_m(p)^{r_m},
\end{equation}
where $d$ is the integer as in~\cite[Table~5.1]{KL1990}. Let $\alpha$ and $\beta$ be the largest and second largest elements in the set $\{m\mid r_m\neq0\}$. Then $(h,\alpha,\beta)$ lies in Table~$\ref{TABLE002}$, and one of the following holds:
\begin{enumerate}[{\rm(a)}]
 \item\label{LEM013.1} $r(G)=p$, $\ell(T)=h$, $\omega(G)=\alpha$, $\psi(G)=\beta$, and the values for $F_1(G)$ and $F_2(G)$ are described in Table~$\ref{TABLE002}$;
\item\label{LEM013.2} $G$ lies in Table~$\ref{TABLE003}$ with $r(G)$, $\ell(G)$, $\omega(G)$, $\psi(G)$, $F_1(G)$ and $F_2(G)$ described there;
\item\label{LEM013.3} $G$ lies in Proposition~$\ref{PROP011}$ with $|G|\leq2^{16}$.
\end{enumerate}
\end{proposition}

\begin{table}[htbp]
\centering
\caption{The Artin invariants of groups $G$ of Lie type over a field of order $q=p^f$}
\vspace{-1ex}
\begin{tabular}{lccccc}
\toprule
$T$  & $h$  &  $\alpha$ &
$\beta$  & $F_1$  & $F_2$ \\
\midrule
$\PSL_{n+1}(q)$, $p>2$ or $(n+1)f\geq9$   & $\frac{(n+1)nf}{2}$   &  $(n+1)f$  &
$nf$  &  $[\frac{n}{2},\frac{n}{2}+\frac{1}{4})$  &  $(\frac{1}{2},1]$\vspace{0.8ex}\\

$\PSL_{n+1}(2^f)$, $(n+1)f\leq8$    & $\frac{(n+1)nf}{2}$   &  $(n+1)f$  &
$nf$  &  $[\frac{n}{2},\frac{n}{2}+\frac{1}{3}]$  &  $[\frac{1}{3},1]$\vspace{0.8ex}\\

$\PSp_{2n}(q)$, $n\geq2$ and $T\neq\PSp_4(4)$  & $n^2f$   &  $2nf$  &
$2(n-1)f$  &  $[\frac{n}{2},\frac{n}{2}+\frac{1}{4})$  &  $(-\frac{1}{2},0]$\vspace{0.8ex}\\

$\PSp_{4}(4)$    & $8$   &  $4f$
& $2f$ & $\frac{5}{4}$  &  $\frac{1}{2}$\vspace{0.8ex}\\

$\PSU_{n}(q)$, $n\geq6$ even    & $\frac{n(n-1)f}{2}$   &  $2(n-1)f$  &
$2(n-3)f$  &  $[\frac{n}{4},\frac{n}{4}+\frac{1}{10}]$  &  $[-\frac{7}{10},-\frac{1}{2}]$\vspace{0.8ex}\\

$\PSU_{4}(q)$    & $6f$   &  $6f$  &
$4f$  &  $1$  &  $1$\vspace{0.8ex}\\

$\PSU_{n+1}(q)$, $n$ even    & $\frac{(n+1)nf}{2}$   &  $2(n+1)f$  &
$2(n-1)f$  &  $[\frac{n}{4},\frac{n}{4}+\frac{1}{6}]$  &  $[\frac{1}{6},\frac{1}{2}]$\vspace{0.8ex}\\

$\POm_{2n+1}(q)$, $n\geq3$    & $n^2f$   &  $2nf$  &
$2(n-1)f$  &  $[\frac{n}{2},\frac{n}{2}+\frac{1}{6}]$  &  $[-\frac{1}{3},0]$\vspace{0.8ex}\\

$\POm^+_{2n}(q)$, $n\geq4$    & $n(n-1)f$   &  $2(n-1)f$  &
$2(n-2)f$  &  $[\frac{n}{2},\frac{n}{2}+\frac{1}{6}]$  &  $[-\frac{4}{3},-1]$\vspace{0.8ex}\\

$\POm^-_{2(n+1)}(q)$, $n\geq3$    & $(n+1)nf$   &  $2(n+1)f$  &
$2nf$  &  $[\frac{n}{2},\frac{n}{2}+\frac{1}{8}]$  &  $[\frac{3}{4},1]$\\

\midrule

${^{2}\BB}_2(2^f)$, $f\equiv3\ (\bmod\ 6)$    & $2f$   &  $4f$  &
$\frac{4f}{3}$  &  $\frac{1}{2}$  &  $\frac{1}{2}$\vspace{0.8ex}\\

${^{2}\BB}_2(2^f)$, $f\equiv\pm1\ (\bmod\ 6)$    & $2f$   &  $4f$  &
$f$  &  $\frac{1}{2}$  &  $\frac{1}{3}$\vspace{0.8ex}\\

${^{3}\D}_4(q)$    & $12f$   &  $12f$  &
$6f$  &  $[1,\frac{13}{12}]$  &  $[-\frac{1 }{6},0]$\vspace{0.8ex}\\

$\G_2(q)$, $q\geq3$    & $6f$   &  $6f$  &
$3f$  &  $[1,\frac{7}{6}]$  &  $[-\frac{1}{3},0]$\vspace{0.8ex}\\

${^{2}\G}_2(3^f)$, $f\geq3$ odd    & $3f$   &  $6f$  &
$2f$  &  $[\frac{1}{2},\frac{7}{12}]$  &  $[\frac{1}{3},\frac{1}{2}]$\vspace{0.8ex}\\

$\F_4(q)$    & $24f$   &  $12f$  &
$8f$  &  $[2,\frac{25}{12}]$  &  $[-\frac{7}{6},-1]$\vspace{0.8ex}\\

${^{2}\F}_4(q)$    & $12f$   &  $12f$  &
$6f$  &  $1$  &  $0$\vspace{0.8ex}\\

$\E_6(q)$    & $36f$   &  $12f$  &
$9f$  &  $[3,\frac{37}{12}]$  &  $[-\frac{13}{6},-2]$\vspace{0.8ex}\\

${^{2}\E}_6(q)$    & $36f$   &  $18f$  &
$12f$  &  $[2,\frac{37}{18}]$  &  $[-\frac{10}{9},-1]$\vspace{0.8ex}\\

$\E_7(q)$    & $63f$   &  $18f$  &
$14f$  &  $[\frac{7}{2},\frac{127}{36}]$  &  $[-\frac{23}{9},-\frac{5}{2}]$\vspace{0.8ex}\\

$\E_8(q)$    & $120f$   &  $30f$  &
$24f$  &  $[4,\frac{241}{60}]$  &  $[-\frac{91}{30},-3]$\\

\bottomrule
\end{tabular}
\label{TABLE002}
\end{table}

\begin{table}[htbp]
\centering
\caption{The Artin invariants for anomalous cases of Lie groups $G$}
\vspace{-1ex}
\begin{tabular}{lcccccc}
\toprule
\ \ \ \ $G$  & $r(G)$ & $\ell(G)$  &  $\omega(G)$
& $\psi(G)$  & $F_1(G)$  & $F_2(G)$ \\
\midrule


$T=\PSL_{2}(2^\ell-1)$, $\ell>7$ & $2$  &  $\ell,\ell+1$ &
$\ell$ & $\ell-1$ &  $[1,1+\frac{1}{\ell}]$  & $[\ell-2-\frac{2}{\ell},\ell-2]$\vspace{0.6ex}\\

$T=\PSL_{2}(2^7-1)$ & $2$  &  $7,8$  &
$7$ & $3$  &  $1,\frac{8}{7}$  &  $-\frac{1}{4},-\frac{15}{28}$\vspace{0.6ex} \\


$T=\PSL_{2}(2^\ell+1)$, $\ell>5$ & $2$  &  $\ell,\ell+1$ &
$2\ell$ & $2(\ell-1)$ &  $[\frac{1}{2},\frac{1}{2}+\frac{1}{2\ell}]$  & $[\ell-1-\frac{1}{\ell},\ell-1]$\vspace{0.6ex} \\


$T=\PSL_{2}(2^f)$, $f>4$ & $2^f+1$  &  $1$  & $f$
& $\frac{f}{2}$  &  $\frac{1}{f}$   &  $\frac{2(f-1)}{f}$\vspace{0.6ex}\\

$\PGU_{3}(8).3$ & $3$  &  $6$ &
$18$ & $6$ & $\frac{1}{3}$   & $\frac{5}{6}$\vspace{0.6ex} \\

$\PGammaL_{2}(49)$ & $2$  & $6$  &
$4$ &  $3$ &  $\frac{3}{2}$  & $1$\vspace{0.6ex} \\

$\PGammaL_{2}(81)$ &  $2$ & $7$  &
$20$ & $4$ &  $\frac{1}{3}$  & $\frac{5}{9}$\vspace{0.6ex} \\

$\PGSp_{4}(3)$ & $2$  & $7$  &
$4$ & $2$ &  $\frac{7}{4}$  & $-\frac{3}{2}$\vspace{0.6ex} \\

$\PSU_{4}(3).\mathrm{D}_8$ & $2$  &  $10$ &
$4$ & $3$  &  $\frac{5}{2}$  & $-1$\\

\midrule


$T=\PSL_{2}(p^2)$, $p=2^\ell-1$ & $p$  & $2$ &
$4$  & $1$  &  $\frac{1}{2}$   & $\frac{1}{3}$\vspace{0.6ex}  \\

$T=\PSL_{3}(2^\ell-1)$ & $p$  & $3$ &
$3$  & $1$  &  $1$   &  $-\frac{1}{2}$\vspace{0.6ex} \\

$T=\PSp_{4}(2^\ell-1)$ & $p$  & $4$  &
$4$  & $1$  &  $1$   &  $-\frac{2}{3}$\vspace{0.6ex} \\

$T=\PSU_{3}(2^\ell-1)$, $\ell>2$ & $p$  & $3$  &
$6$  & $1$  &  $\frac{1}{2}$   &  $\frac{1}{5}$\\

\midrule


$T=\PSL_{3}(4)$ & $2$  & $6,7,8$  &
$4$ & $3$ &  $\frac{3}{2},\frac{7}{4},2$  & $1,\frac{1}{2},0$\vspace{0.6ex}\\

$T=\PSL_{5}(2)$ & $2$  &  $10,11$ &
$5$ & $4$ &  $2,\frac{11}{5}$  & $1,\frac{3}{5}$\vspace{0.6ex}\\

$T=\PSp_{6}(2)$ & $2$  & $9$  &
$4$ & $3$  &  $\frac{9}{4}$  & $-\frac{1}{2}$\vspace{0.6ex}\\

$T=\POm_{8}^+(2)$ & $2$  & $12,13$  &
$4$ & $3$ & $3,\frac{13}{4}$ & $-2,-\frac{5}{2}$\\

%

\midrule
%

$T=\PSL_{2}(64)$ & $2$  & $6,7$  &
$12$ & $4$  & $\frac{1}{2},\frac{4}{7}$ & $\frac{1}{2},\frac{1}{3}$\vspace{0.6ex} \\

$T=\PSL_{3}(8)$ & $2$ & $9,10$  &
$9$ & $3$ & $1,\frac{10}{9}$ & $-\frac{1}{2},-\frac{5}{7}$\vspace{0.6ex}\\

$T=\PSL_{4}(4)$ & $2$  &  $12,13,14$ &
$8$ & $4$ & $\frac{3}{2},\frac{13}{8},\frac{7}{4}$  & $-1,-\frac{5}{4},-\frac{3}{2}$\vspace{0.6ex} \\

$T=\PSL_{7}(2)$ & $2$ & $21,22$  &
$7$ & $5$  & $3,\frac{22}{7}$  & $-\frac{5}{2},-\frac{25}{9}$\vspace{0.6ex} \\

$T=\PSp_4(8)$ & $2$  & $12,13$ &
$12$ & $4$ &  $1,\frac{10}{9}$  & $-\frac{1}{2},-\frac{2}{3}$\vspace{0.6ex} \\

$T=\PSp_8(2)$ &  $2$ & $16$  &
$8$ & $4$  & $2$   & $2$\vspace{0.6ex}\\

$\PSU_3(8).[a]$, $a\mid3$ & $2$ & $9$  &
$18$ & $3$ &  $\frac{1}{2}$  & $\frac{1}{5}$\vspace{0.6ex}\\

$\PSU_3(8).[2a]$, $a\mid9$ & $2$ & $10$  &
$18$ & $3$ &  $\frac{5}{9}$  & $\frac{1}{9}$\vspace{0.6ex}\\

$T=\PSU_5(2)$ & $2$ & $10,11$  &
$10$ & $4$  &  $1,\frac{10}{9}$ & $-\frac{1}{3},-\frac{5}{9}$\vspace{0.6ex}\\

$T=\PSU_6(2)$ & $2$  & $15,16$  &
$10$ &  $4$ & $\frac{3}{2},\frac{8}{5}$ & $-\frac{4}{3},-\frac{14}{9}$\vspace{0.6ex}\\

$T=\POm_{8}^-(2)$ & $2$  & $12,13$  &
$8$ & $4$  &  $\frac{3}{2},\frac{13}{8}$ & $-1,-\frac{5}{4}$\vspace{0.6ex}\\

$T=\POm_{10}^+(2)$ & $2$  & $20,21$  &
$8$ & $5$ &  $\frac{5}{2},\frac{21}{8}$ & $-\frac{7}{3},-\frac{18}{7}$\vspace{0.6ex}\\

$T={^{3}\D}_4(2)$ & $2$  & $12$  &
$12$ & $3$  & $1$   &  $-\frac{2}{3}$\vspace{0.6ex}\\

$T=\G_2(4)$ &  $2$ & $12,13$  &
$12$ & $4$  &  $1,\frac{13}{3}$  & $-\frac{1}{2},-\frac{43}{6}$\vspace{0.6ex}\\

$T={^{2}\F}_4(2)'$ & $2$  & $11,12$  &
$12$ & $4$ &  $\frac{11}{12},1$  & $-\frac{1}{3},-\frac{1}{2}$\\

\bottomrule
\end{tabular}
\label{TABLE003}
\end{table}

\begin{proof}
For every simple group $T$ as listed in the first column of Table~\ref{TABLE002}, the quantities $h$, $\alpha$, $\beta$ from~\eqref{eq0.6} are provided in the next three columns of Table~\ref{TABLE002} (see~\cite[Table~A.1]{KLST1990}). In particular, both $\alpha$ and $\beta$ are greater than or equal to any prime factor of $|\Out(T)|/d$.

Assume that $r(G)\neq p$. Then, it follows from Propositions~\ref{PROP011} and~\ref{PROP012} that $G$ lies in Proposition~\ref{PROP011}. Moreover, in the case where $|G|>2^{16}$, namely, not covered by~\eqref{LEM013.3}, all possible $G$ and their Artin variants are listed in the first nine rows of Table~\ref{TABLE003}. Based on these, it is straightforward to verify that $F_1(G)$ and $F_2(G)$ are as described in Table~\ref{TABLE003} as well.

Assume that $r(G)=p$. If there is no primitive prime divisor of $(p,\alpha)$, then we deduce from Lemma~\ref{lem0.10} that either $p$ is a Mersenne prime and $\alpha=2$ or $(p,\alpha)=(2,6)$. The former case is excluded by the value of $\alpha$ in Table~\ref{TABLE002}, while for the latter case, all possible $G$ and their Artin variants (and hence the values of $F_1(G)$ and $F_2(G)$) are provided in Table~\ref{TABLE003}. Now suppose that $(p,\alpha)$ has a primitive prime divisor, and in particular, $\omega(G)\geq\alpha$. If $\omega(G)>\alpha$, then there is a primitive prime divisor $r$ of $(p,\omega(G))$ such that $r$ does not divide any $\Phi_m(p)$ in~\eqref{eq0.6}. This implies that $r$, with $r>\omega(G)>\alpha$, is a prime factor of $|\Out(T)|/d$, a contradiction. Hence, $\omega(G)=\alpha$ as claimed in~\eqref{LEM013.1}. The conclusion that $\psi(G)=\beta$ except for the groups described in Table~\ref{TABLE003} can be obtained through a similar argument.
To finish the proof, it remains to estimate $F_1(G)$ and $F_2(G)$ for $G$ in Table~\ref{TABLE002}. For these groups $G$, since $d$ and $p$ are coprime and $\ell(G)\leq\ell(T)+\log_p(|\Out(T)|_p)$, we deduce that
\[
\frac{h}{\alpha}=\frac{\ell(T)}{\omega(G)}\leq \frac{\ell(G)}{\omega(G)}=F_1(G)\leq\frac{h}{\alpha}+\frac{\log_p(|\Out(T)|_p)}{\alpha}
=\frac{h}{\alpha}+\frac{\log_p(|\Out(T)|/d)_p}{\alpha}.
\]
Take $h/\alpha$ as the lower bound of the interval for $F_1(G)$, and estimate $\log_p(|\Out(T)|/d)_p$ to derive an upper bound of $F_1(G)$. We then obtain the intervals for $F_1(G)$ and consequently for $F_2(G)$, as shown in Table~\ref{TABLE002}.
\end{proof}

\subsection{Proof of Proposition~\ref{THM004}}\label{subsec5.3}

Based on the results in Subsections~\ref{SUBSEC5.1} and~\ref{subsec5.2}, we are now ready to establish Proposition~\ref{THM004}.
\begin{proof}[Proof of Proposition~$\ref{THM004}$]
Computation in \textsc{Magma}~\cite{BCP1997} shows that apart from the stated exceptions, there are no two almost simple groups with the same order less than $2^{16}$. In the following, assume $|H|=|K|>2^{16}$, and let $S$ and $T$ be the socle of $H$ and $K$, respectively. Recall the definition of the logarithmic proportion $\lambda$ in Subsection~\ref{subsec5.2}, as well as Definition~\ref{DEF002} for Artin invariants and functions $F_1$ and $F_2$. The assumption $|H|=|K|$ implies that $H$ and $K$ share the same logarithmic proportion and Artin invariants.

Assume that $H$ has alternating socle $\Al_n$. As $|H|>2^{16}$, we have $n\geq9$. If $K$ has alternating or sporadic socle, then it is clear (by the list of sporadic simple groups) that $S$ and $T$ must be isomorphic. So suppose that $K$ is of Lie type. If $n>20$, we derive from Lemma~\ref{LEM008} that
\[
2^{60}<|H|=|K|\leq|T|\cdot|\Out(T)|\leq|T|\cdot\log_2|T|
\]
and so $|T|>2^{54}$. Then it follows from Lemmas~\ref{LEM010} and~\ref{LEM007} that $H$ and $K$ cannot have the same order since $\lambda(H)\neq\lambda(K)$. When $9\leq n\leq20$, the Artin invariants for all possibilities of $H$ are lists in Table~\ref{TABLE005} (the Artin invariants of small alternating groups from~\cite[Table~A.4]{KLST1990} could help to obtain our table). Examining Tables~\ref{TABLE002} and~\ref{TABLE003} for the Artin invariants of $K$, as claimed in Proposition~\ref{LEM013}, we conclude that
\[
(S,T)\in\{(\Al_9,\PSL_3(4)),\,(\Al_{10},\PSL_3(4)),\,(\Al_{20},\PSL_3(64)),\,(\Al_{20},\PSU_4(8))\}
\]
are the only cases where $H$ and $K$ share the same Artin invariants, but none satisfies $|H|=|K|$.

\begin{table}[htbp]
\centering
\caption{Artin invariants of small alternating and symmetric groups}
\vspace{-1ex}
\begin{tabular}{lcccccc}
\toprule
$H$ & $r(H)$ & $\ell(H)$ & $\omega(H)$ & $\psi(H)$ & $F_1(H)$ & $F_2(H)$\\
\midrule

$\Al_9$ & $3$   & $4$   &  $6$  &
$4$  &  $\frac{1}{3}$  &  $\frac{7}{3}$\vspace{0.6ex}\\

$\Sy_9$, $\Al_{10}$ & $2$    & $7$   &  $4$  &
$3$  &  $\frac{7}{4}$  &  $\frac{1}{2}$\vspace{0.6ex}\\

$\Sy_{10}$ & $2$   & $8$   &  $4$  &
$3$  &  $2$  &  $0$\vspace{0.6ex}\\

$\Al_{11}$ & $2$   & $7$   &  $10$  &
$4$  &  $\frac{7}{10}$  &  $\frac{4}{15}$\vspace{0.6ex}\\

$\Sy_{11}$ & $2$   & $8$   &  $10$  &
$4$  &  $\frac{4}{5}$  &  $\frac{1}{15}$\vspace{0.6ex}\\

$\Al_{12}$ & $2$   & $9$   &  $10$  &
$4$  &  $\frac{9}{10}$  &  $-\frac{2}{15}$\vspace{0.6ex}\\

$\Sy_{12}$ & $2$   & $10$   &  $10$  &
$4$  &  $1$  &  $-\frac{1}{3}$\vspace{0.6ex}\\

$\Al_n,\Sy_n$ with $13\leq n\leq19$ & $2$   &  $[9,16]$   &  $12,18$   &
$10,12$  &   $[\frac{3}{4},\frac{4}{3}]$  & $[\frac{10}{9},\frac{9}{2}]$\vspace{0.6ex}\\

$\Al_{20}$ & $2$   & $17$   &  $18$  &
$12$  &  $\frac{17}{18}$  &  $\frac{10}{9}$\vspace{0.6ex}\\

$\Sy_{20}$ & $2$   & $18$   &  $18$  &
$12$  &  $1$  &  $1$\\

\bottomrule
\end{tabular}
\label{TABLE005}
\end{table}

Assume that $H$ is of Lie type and that $K$ has sporadic socle $T$. One can then obtain the Artin invariants as well as $F_1(K)$ and $F_2(K)$ for each $K$ through a direct calculation (the Artin invariants of $T$, listed in Table~\cite[Table~A.5]{KLST1990}, might be helpful). Similarly, by comparing Artin invariants and the values of the functions $F_1$ and $F_2$ of $K$ with those of $H$, listed in Tables~\ref{TABLE002} and~\ref{TABLE003}, we conclude that the only possible pairs $(S,T)$ are $(\PSL_3(4),\J_2)$ and $(\PSU_5(2),\HS)$. However, neither pair satisfies $|H|=|K|$.

Clearly, the orders of $H$ and $K$ cannot be equal when both $S$ and $T$ are sporadic groups. To complete the proof, suppose that both $H$ and $K$ are of Lie type. Then the Artin invariants and values of $F_1$ and $F_2$ on $H$ and $K$ are listed in Table~\ref{TABLE002} or~\ref{TABLE003}. By careful examination of these values, we present in Table~\ref{TABLE004} the tuples $(X_1,X_2,\ldots)$ such that there exist almost simple groups of Lie type sharing the same Artin invariants with socles $X_1,X_2,\ldots$ and $|X_1|\geq|X_2|\geq\cdots$. Moreover, the $i$-th entry in the last column of Table~\ref{TABLE004} gives a lower bound for $|X_i|/|X_{i+1}|$.
Direct computation shows that every lower bound of $|X_i|/|X_{i+1}|$ is greater than $|\Out(X_{i+1})|$, except when $\{X_1,X_2\}$ is equal to one of:
\[
\{\PSp_{2n}(q),\POm_{2n+1}(q)\},\ \{\PSL_3(q^2),\PSU_4(q)\},\ \{\PSp_4(q^3),{^{3}\D}_4(q)\},\ \{\G_2(q^2),{^{3}\D}_4(q)\},
\]
where $\{\PSL_3(q^2),\PSU_4(q)\}$ satisfies $q\leq16$. However, none of the last three cases can yield two almost simple groups of the same order, either due to a straightforward calculation or the existence of primitive prime divisor of $(q,6)$. We conclude that $\{S,T\}=\{\PSp_{2n}(q),\POm_{2n+1}(q)\}$, as required by the proposition.
\end{proof}

\vspace{-1ex}
\begin{table}[htbp]
\centering
\caption{Almost simple groups of Lie type with the same Artin invariants}
\vspace{-1ex}
\begin{tabular}{ll@{\hskip -10pt}r}
\toprule
$(F_1(X_i),F_2(X_i))$ & tuple $(X_1,X_2,\ldots)$ of socles & $|X_i|/|X_{i+1}|>$ \vspace{0ex}\\
\midrule

\ \ $(1,1)$ & $(\PSL_3(q^2),\PSU_4(q))$ &  $q/(3,q^2-1)$
\vspace{0.5ex} \\

\ \ $(n/2,1)$ & $(\PSL_n(q^2),\POm_{2n}^-(q))$, $n\geq4$ & $q^{n-2}/(n,q^2-1)$
\vspace{0.5ex} \\

\ \ $(n/2,0)$ & $(\PSp_{2n}(q),\POm_{2n+1}(q))$, $n\geq3$ &
\vspace{0.5ex} \\

\ \ $(1,0)$& $(\PSp_4(q^3),\G_2(q^2),{^{3}\D}_4(q),{^{2}\F}_4(q))$ & $q^2/2$,\ ---\;,\ $q$
\vspace{0.5ex} \\

\ \ $(1/2,1/2)$ & $(\PSU_3(q^2),{^{2}\BB}_2(q^3))$, $(\PSU_3(q),{^{2}\G}_2(q))$ & $|\Aut(X_2)|$
\vspace{0.5ex} \\

\ \ $(2,-1)$& $(\POm_8^+(q^2),\F_4(q))$, $(\POm_8^+(q^3),{^{2}\E}_6(q))$, $({^{2}\E}_6(q^2),\F_4(q^3))$ & $q^3/4$
\vspace{0.5ex} \\

\ \ $(3/2,1)$& $(\PSL_2(49),\PSL_3(4),\PSL_4(2))$ & ---, ---
\vspace{0.5ex} \\

\ \ $(1/2,1/2)$& $(\PSL_2(64),\PSU_3(4))$ & 4
\vspace{0.5ex} \\

\ \ $(3/2,-1)$& $(\PSL_4(4),\POm_8^-(2))$ & 4
\vspace{0.5ex}\\

\ \ $(1,-1/2)$& $(\PSp_4(8),\G_2(4))$ & 2
\vspace{0ex}\\

\bottomrule
\end{tabular}
\label{TABLE004}
\end{table}

\section{Primitive permutation representations of $\Al_n$ and $\Sy_n$ of the same degree}\label{SEC3}

The aim of this section is to establish Proposition~\ref{THM002}. We prove Proposition~\ref{THM002} in Subsection~\ref{SUBSEC3.4} after preparations in Subsections~\ref{SUBSEC3.1} and~\ref{SUBSEC3.2}. Recall that for an integer $n$, its $p$-part $n_p$ is the largest power of $p$ dividing $n$. The \emph{$p$-adic valuation} $v_p(n)$ of $n$ is defined as the exponent of $p$ in $n_p$, namely, $v_p(n)=\log_p(n_p)$.

\subsection{Number-theoretic results}\label{SUBSEC3.1}

The first lemma of this subsection is Bertrand's postulate, and we refer to~\cite[Charpter~2]{TheBook} for a proof.
\begin{lemma}\label{lem0.4}
For every integer $n\geq3$, there exists a prime $p$ such that $n/2<p<n$.
\end{lemma}

The following lemma, known as Stirling's Formula~\cite{Robbins1955}, provides bounds for the factorial.
\begin{lemma}\label{lem0.2}
Let $x$ be a positive integer. Then the factorial $x!$ satisfies
\[
\sqrt{2\pi x}\left(\frac{x}{e}\right)^x <x!<\sqrt{2\pi x}\left(\frac{x}{e}\right)^x e^{\frac{1}{12x}}.
\]
\end{lemma}

With the aid of Lemma~\ref{lem0.2}, we obtain the following lemma.

\begin{lemma}\label{LEM011}
For a set of size $n$, the number of its subsets with a given size is at most $2^{n}/\sqrt{n}$.
\end{lemma}

\begin{proof}
Clearly, the lemma holds true for $n=1$, and so we assume $n\geq2$. According to Lemma~\ref{lem0.2}, for any positive integer $x$, the factorial $x!$ satisfies
\[
\sqrt{2\pi x}\left(\frac{x}{e}\right)^x <x!<\sqrt{2\pi x}\left(\frac{x}{e}\right)^x e^{\frac{1}{12x}}.
\]
Hence,
\[
\binom{2x}{x}=\frac{(2x)!}{x!\cdot x!}<
       e^{\frac{1}{24x}}
\frac  {\sqrt{4\pi x}\cdot (2x)^{2x}}
       {\big(\sqrt{2\pi x} \cdot x^{x}\big)^2}=\frac{e^{\frac{1}{24x}}}{\sqrt{\pi}}\cdot \frac{2^{2x}}{\sqrt{x}}
       < \frac{2^{2x}}{\sqrt{2x}}.
\]
Note that the maximum binomial coefficient in $\binom{n}{0},\ldots,\binom{n}{n}$ is $\binom{n}{\lfloor n/2\rfloor}$. If $n$ is even, the conclusion immediately follows from the above inequality by taking $2x$ as $n$. If $n$ is odd, we also obtain
\[
\binom{n}{\lfloor n/2\rfloor}=\frac{n}{(n+1)/2}\binom{n-1}{(n-1)/2}
<\frac{2n}{n+1}\cdot\frac{2^{n-1}}{\sqrt{n-1}}<\frac{2^{n}}{\sqrt{n}},
\]
as the lemma asserts.
\end{proof}

We also need the following result on distribution of primes (see, for example,~\cite[(3.6)]{RS1962}).

\begin{lemma}\label{lem0.3}
Let $n$ be a real number greater than $1$. Then the number of primes less than or equal to $n$ is smaller than $1.25506n/\ln n$.
\end{lemma}

The Diophantine equation $\binom{n+k}{k}=2\binom{n}{k}$ arises as an crucial case in the proof of Proposition~\ref{THM002}. We show that this equation only has the trivial solution $n=k=1$.

\begin{lemma}\label{lem0.1}
There equation $\binom{n+k}{k}=2\binom{n}{k}$ has no positive integer solutions with $\{n,k\}\neq\{1\}$.
\end{lemma}

\begin{proof}
The lemma clearly holds when $k=1$. So we let $k\geq2$, and let
\begin{equation}\label{eq0.1}
f(n,k)=\frac{\binom{n+k}{k}}{\binom{n}{k}}=\frac{(n+k)(n+k-1)\cdots(n+1)}{n(n-1)\cdots(n-k+1)}.
\end{equation}
Suppose for a contradiction that there exist some integers $n$ and $k$ satisfying $f(n,k)=2$.

Since $k\geq2$, we derive the following bounds for $f(n,k)$ from~\eqref{eq0.1}:
\[
\left(\frac{n+k}{n}\right)^k<f(n,k)< \left(\frac{n+1}{n-k+1}\right)^k.
\]
Substituting $f(n,k)=2$, we obtain $(n+k)/n<2^{1/k}<(n+1)/(n-k+1)$, and so
\begin{equation*}
\frac{k}{2^{\frac{1}{k}}-1}<n<k\left(\frac{1}{2^{\frac{1}{k}}-1}+1\right)-1.
\end{equation*}
Note that $k<1/(2^{\frac{1}{k}}-1)<2k-1$ for $k\geq2$. We then conclude
\begin{equation}\label{eq0.3}
k^2<n<2k^2.
\end{equation}
For the case $k<152$, as $n<2k^2< 46208$, it can be checked by computer that the only pair $(n,k)$ such that $f(n,k)=2$ is $(n,k)=(1,1)$. In the rest of the proof, assume that $k\geq 152$.

If there exists a prime divisor $p$ of $n(n-1)\cdots(n-k+1)$ with $p>2k$, then none of the integers $n+1,n+2,\ldots,n+k$ is divisible by $p$, which contradicts $f(n,k)=2$ by~\eqref{eq0.1}. Thus, since $n(n-1)\cdots(n-k+1)$ is divisible by $\binom{n}{k}$, we derive that no prime factor of $\binom{n}{k}$ is greater than $2k$. As a consequence,
\[
\binom{n}{k}=\prod_{p\leq 2k} {\binom{n}{k}}_p.
\]
By Legendre's formula~(see, for example,~\cite[Theorem~2.6.7]{Moll2012}),
\begin{align*}
v_p\left({\binom{n}{k}}\right)&=v_p(n!)-v_p(k!)-v_p((n-k)!)
\\
&=\sum_{j=1}^{\lfloor\log_pn\rfloor}\left(\left\lfloor\frac{n}{p^j}\right\rfloor-\left\lfloor\frac{k}{p^j}\right\rfloor- \left\lfloor\frac{n-k}{p^j}\right\rfloor\right)\leq\sum_{j=1}^{\lfloor\log_pn\rfloor}1 \leq \log_pn
\end{align*}
for any prime $p$. Hence, we deduce from Lemma~\ref{lem0.3} that
\begin{equation*}
\binom{n}{k}=\prod_{p\leq 2k} {\binom{n}{k}}_p\leq \prod_{p\leq 2k}n< n^{\frac{2.51012k}{\ln(2k)}}.
\end{equation*}
Combining this with~\eqref{eq0.3}, we establish
\begin{equation}\label{eq0.2}
\binom{k^2}{k}<\binom{n}{k}< \big(2k^2\big)^{\frac{2.51012k}{\ln(2k)}}.
\end{equation}

It follows from Lemma~\ref{lem0.2} that
\begin{align*}
\binom{k^2}{k} = \frac{\big(k^2\big)!}{k!\big(k^2-k\big)!} & \;
       > \;
\frac  {\sqrt{2\pi k^2}}
       {\sqrt{2\pi k}\cdot \sqrt{2\pi(k^2-k)}}
       \cdot
\frac  {\big(k^2\big)^{k^2}}
       {k^k \cdot \big(k^2-k\big)^{k^2-k}}
       \cdot
       e^{-\frac{1}{12k}-\frac{1}{12(k^2-k)}}
       \nonumber\\
       & \; = \;
\frac  {1}
       {\sqrt{2\pi(k-1)}}   \cdot
\frac  {k^{k^2}}
       {(k-1)^{k^2-k}}
       \cdot
       e^{-\frac{1}{12k}-\frac{1}{12(k^2-k)}}
       \nonumber\\
       & \; > \;
\frac  {1}
       {\sqrt{2\pi k}}   \cdot
\frac  {k^{k^2}}
       {(k-1)^{k^2-k}}
       \cdot
       e^{-2}.
\end{align*}
Substituting this into~\eqref{eq0.2}, we obtain
\begin{equation*}
\frac  {1}
       {\sqrt{2\pi k}}   \cdot
\frac  {k^{k^2}}
       {(k-1)^{k^2-k}}
       \cdot
       e^{-2}
       <
       \big(2k^2\big)^{\frac{2.51012k}{\ln(2k)}},
\end{equation*}
which implies
\begin{align*}
k\ln(k-1)+k^2\ln\left(1+\frac{1}{k-1}\right) & \; < \; \frac{2.51012\ln(2k^2)}{\ln(2k)}\cdot k+\frac{\ln(2\pi k)}{2}+2<6k+3.
\end{align*}
This yields
\begin{align*}
k\ln(k-1)+k^2\left(\frac{1}{k-1}-\frac{1}{2}\Big(\frac{1}{k-1}\Big)^2\right) <6k+3,
\end{align*}
that is,
\[
\ln(k-1)<\frac{6k+3}{k}-\frac{k}{k-1}+\frac{k}{2(k-1)^2}.
\]
This is a contradiction to $k\geq152$, completing the proof.
\end{proof}

Recall that $v_p(n)$ is defined as the exponent of $p$ in $n_p$.

\begin{lemma}\label{lem0.5}
Let $n$ be a positive integer. Then for each prime $p$, the following statements hold:
\begin{enumerate}[{\rm(a)}]
\item\label{lem0.5.1} $v_p(n!)\leq(n-1)/(p-1)$;
\item\label{lem0.5.2} $v_p((2n)!/n!)\leq n$, and equality holds if and only if $p=2$.
\end{enumerate}
\end{lemma}

\begin{proof}
Let $n=\sum_{i=0}^{k}a_ip^i$ be the $p$-adic expansion, where $0\leq a_i\leq p-1$ for each $i\leq k$. Then
\[
v_p(n!)=\sum_{j=1}^{\lfloor\log_pn\rfloor}\left\lfloor\frac{n}{p^j}\right\rfloor
=\sum_{j=1}^{k}\sum_{i=j}^{k}a_ip^{i-j}.
\]
Hence, we conclude that
\begin{align*}
(p-1)\cdot v_p(n!)&=\sum_{j=1}^{k}\sum_{i=j}^{k}a_ip^{i-j+1}-\sum_{j=1}^{k}\sum_{i=j}^{k}a_ip^{i-j}\\
&=\left(\sum_{i=1}^{k}a_ip^{i}+\sum_{j=2}^{k}\sum_{i=j}^{k}a_ip^{i-j+1}\right)
-\left(\sum_{j=1}^{k-1}\sum_{i=j+1}^{k}a_ip^{i-j}+\sum_{j=1}^{k}a_i\right)\\
&=\sum_{i=1}^{k}a_ip^{i}-\sum_{j=1}^{k}a_i=n-\sum_{j=0}^{k}a_i\leq n-1,
\end{align*}
as statement~\eqref{lem0.5.1} asserts. Consequently, if $p\geq3$, then $v_p((2n)!/n!)\leq v_p((2n)!)\leq(2n-1)/2<n$. If $p=2$, then
\[
v_p((2n)!/n!)=v_2((2n)!)-v_2(n!)=
\sum_{j=1}^{\lfloor\log_2n\rfloor+1}\left\lfloor\frac{n}{2^{j-1}}\right\rfloor
-\sum_{j=1}^{\lfloor\log_2n\rfloor}\left\lfloor\frac{n}{2^j}\right\rfloor=n,
\]
which completes the proof.
\end{proof}

The next lemma is straightforward to verify, so we omit the proof here.

\begin{lemma}\label{lem0.6}
Let $a$ and $b$ be positive integers with $a>b$. Then the following statements hold:
\begin{enumerate}[{\rm(a)}]
\item\label{lem0.6.1} $|\Sy_a\wr\Sy_b|>|\Sy_b\wr\Sy_a|$;
\item\label{lem0.6.2} if $x$ and $y$ are positive integers with $ab=xy$ and $a>x>y$, then $|\Sy_a\wr\Sy_b|>|\Sy_x\wr\Sy_y|$.
\end{enumerate}
\end{lemma}

We close this subsection with the following two lemmas.

\begin{lemma}\label{lem0.7}
Let $a$, $b$, $x$ and $y$ be positive integers such that $ab=xy$ and $|\Sy_a\wr\Sy_b|=|\Sy_x\wr\Sy_y|$. Then $a=x$ and $b=y$.
\end{lemma}

\begin{proof}
If $\max\{a,b,x,y\}\in\{a,x\}$, then the conclusion immediately follows from Lemma~\ref{lem0.6}. Without loss of generality, we only need to prove that neither $b>x>y>a$ nor $b>y>x>a$. Write $M=(a!)^b\cdot b!=|\Sy_a\wr\Sy_b|$ and $N=(x!)^y\cdot y!=|\Sy_x\wr\Sy_y|$.

Assume that $b>x>y>a$. If $b\geq2a$, then by Lemma~\ref{lem0.4}, there exists a prime $p$ such that $a\leq b/2<p<b$ and so $v_p(M)=1$. This, together with the fact that $v_p(N)=0$ or $v_p(N)\geq y$, contradicts the condition $M=N$. Thus, $b<2a$ and we deduce from $(x-a)y=a(b-y)$ that
\[
1=\frac{M}{N}=\frac{(a!)^b\cdot b!}{(x!)^y\cdot y!}
=\frac{(a!)^{b-y}\cdot b\cdots(y+1)}{(x\cdots(a+1))^y}
<\frac{(a!)^{b-y}\cdot b^{b-y}}{a^{(x-a)y}}
=\left(\frac{a!\cdot b}{a^a}\right)^{b-y}
<\left(\frac{2a!}{a^{a-1}}\right)^{b-y}.
\]
This inequality holds only if $a\leq3$ and so $b<2a\leq6$, which gives rise to finitely many possibilities for $(a,b,x,y)$. Then one may directly check for each possibility that $M\neq N$, a contradiction.

Assume that $b>y>x>a$. If $b\geq 2y$, then Lemma~\ref{lem0.4} gives a prime $p$ with $y<p<b$, so that $p$ divides $M$ but not $N$, a contradiction. Therefore, $b<2y$. Consider a prime $r$ with $x/2<r<x$ (which exists, again by Lemma~\ref{lem0.4}). Then $v_r(N)=y+v_r(y!)$. Note that $v_r(M)=b\cdot v_r(a!)+v_r(b!)$. Since $b>y$ and then $v_r(b!)\geq v_r(y!)$, we deduce from $v_r(N)=v_r(M)$ that $v_r(b!)=v_r(M)=v_r(N)=y+v_r(y!)$. However, since $b<2y$, we derive from Lemma~\ref{lem0.5}\eqref{lem0.5.2} that $v_r(b!)-v_r(y!)<y$, a contradiction. The proof is complete.
\end{proof}

\begin{lemma}\label{lem0.9}
Let $a$, $b$, $x$ and $y$ be positive integers with $a\geq5$, $b\geq2$, $x\geq5$, $y\geq2$ such that $a^b=x^y$ and $a\neq x$. Then $|\Sy_a\wr\Sy_b|\neq|\Sy_x\wr\Sy_y|$ and $|\Sy_a\wr\Sy_b|\neq2|\Sy_x\wr\Sy_y|$.
\end{lemma}

\begin{proof}
Without loss of generality, let $a>x$ and so $b<y$. Then it suffices to show that $M:=|\Sy_a\wr\Sy_b|$ cannot be equal to $N:=i|\Sy_x\wr\Sy_y|$, where $i\in\{1/2,1,2\}$. Suppose on the contrary that $M=N$. It follows from $a^b=x^y$ that $b\cdot v_r(a)=y\cdot v_r(x)$ for each prime $r$, and so $v_r(a)\geq v_r(x)$ for each prime $r$. This, together with $a>x$, implies that $a\geq2x$. By Lemma~\ref{lem0.4}, there exists a prime $p$ such that $a>p>a/2\geq x\geq5$. Hence,
\begin{equation}\label{eq0.5}
b=v_p(M)-v_p(b!)=v_p(N)-v_p(b!)=v_p(y!/b!).
\end{equation}
Combining this with Lemma~\ref{lem0.5}\eqref{lem0.5.2}, we obtain $y>2b$. Moreover, as $v_p(y!)\geq v_p(y!/b!)=b\geq2$ and $p>a/2$, we have $y>a$. If $b\leq3$, then $a^b\leq a^3<5^a\leq x^a<x^y$, a contradiction. Thus, $b\geq4$, and we derive from~\eqref{eq0.5} that $v_p(y!)\geq v_p(y!/b!)=b\geq4$. Since $p>5$, it follows that $y\geq4p>2a$. This, together with $y>2b$, implies that $M=|\Sy_a\wr\Sy_b|$ is not divisible by any prime larger than $y/2$. However, Lemma~\ref{lem0.4} implies that $M=N=i|\Sy_x\wr\Sy_y|$ is divisible by some prime in $(y/2,y)$, a contradiction.
\end{proof}

\subsection{The subgroups of almost simple groups with special order}\label{SUBSEC3.2}

In this subsection, we collect some results on the order of subgroups of almost simple groups.

\begin{lemma}\cite[Corollary~1.2]{Maroti2002}\label{lem000}
Let $G$ be a primitive group of degree $n$ with $G\ngeq\Al_n$. Then $|G|<3^n$. Moreover, if $n>24$, then $|G|<2^n$.
\end{lemma}

The following result exhibits simple groups which have a subgroup of prime power index.

\begin{lemma}\cite[Theorem~1]{Guralnick1983}\label{lem0.0}
Let $T$ be a nonabelian simple group with a subgroup $L$ of prime power index $n$. Then one of the following holds:
\begin{enumerate}[{\rm(a)}]
\item\label{lem0.0.1} $T=\Al_n$ and $L=\Al_{n-1}$;
\item\label{lem0.0.2} $T=\PSL_m(q)$ and $L$ is the stabilizer of a line or hyperplane with $n=(q^m-1)/(q-1)$;
\item\label{lem0.0.3} either $(T,L)\in\{(\PSL_2(11),\Al_5),(\M_{23},\M_{22}),(\M_{11},\M_{10})\}$, or $T=\PSp_4(3)$ and $L$ is the parabolic subgroup of index $27$.
\end{enumerate}
\end{lemma}

The next result identifies almost simple groups $K$, with a subgroup of order divisible by all prime divisors of $|\Soc(K)|$.

\begin{lemma}\cite[Corollary~5]{LPS2000}\label{lem0.8}
Let $K$ be an almost simple group with socle $T$. Suppose that $L$ is a subgroup of $K$ not containing $T$ such that every prime divisor of $|T|$ divides $|L|$. Then all pairs $(T,L)$ are listed in~\cite[Table~10.7]{LPS2000}. In particular, the infinite families for $(T,L)$ with $L$ maximal in $K$ are listed in Table~$\ref{TABLE001}$.
\end{lemma}

\vspace{-1ex}
\begin{table}[H]
\centering
\caption{Infinite families of almost simple $K$ with a maximal subgroup $L$,\\ whose order is divisible by all prime divisors of $|T|$, where $T=\Soc(K)$}
\vspace{-1ex}
\begin{tabular}{llll}
\toprule
$(T,L\cap T)$               & $ |T|$  &
$|T\,{:}\,L\cap T|$         & Conditions \\
\midrule
$(\Al_m,(\Sy_d\times\Sy_{m-d})\cap\Al_m)$ & $m!/2$ &
$\binom{m}{d}$ & $d<m/2$ \\

\rule{0pt}{3ex}$(\PSp_{2m}(q),\OO_{2m}^-(q))$ & $q^{m^2}\prod_{i=1}^m\left(q^{2i}-1\right)$ &
$q^m(q^m-1)/2$ & $m\geq2$ even, $q$ even\\

\rule{0pt}{3ex}$(\POm_{2m+1}(q),\Om_{2m}^-(q).2)$ & $\frac{1}{2}q^{m^2}\prod_{i=1}^m\left(q^{2i}-1\right)$ &
$q^m(q^m-1)/2$ & $m\geq4$ even, $q$ odd\\

\rule{0pt}{3ex}$(\POm_{2m}^+(q),\Om_{2m-1}(q))$ & $\frac{q^{m(m-1)}}{4(q^m+1)}\prod_{i=1}^{m}\left(q^{2i}-1\right)$ &
$q^{m-1}(q^m-1)/2$ & $m\geq4$ even, $q$ odd \\

\rule{0pt}{3ex}$(\POm_{2m}^+(q),\Sp_{2m-2}(q))$ & $\frac{q^{m(m-1)}}{q^m+1}\prod_{i=1}^{m}\left(q^{2i}-1\right)$ &
$q^{m-1}(q^m-1)$ & $m\geq4$ even, $q$ even \\

\rule{0pt}{3ex}$(\PSp_{4}(q),\PSp_{2}(q^2){:}2)$ & $\frac{1}{d}q^4(q^4-1)(q^2-1)$ &
$q^2(q^2-1)/2$ & $d=\gcd(2,q-1)$ \\

\rule{0pt}{3ex}$(\PSL_{2}(p),p{:}((p-1)/2))$ & $p(p^2-1)/2$ &
$p+1$ & $p$ a Mersenne prime\\
\bottomrule
\end{tabular}
\label{TABLE001}
\end{table}

We close this subsection with the following result analyzing the groups in Table~\ref{TABLE001}.

\begin{lemma}\label{CORO001}
Let $K$ be an almost simple group with socle $T$ and a maximal subgroup $L$ as described in Table~$\ref{TABLE001}$, and let $n=|T\,{:}\,T\cap L|$. Then the following statements hold:
\begin{enumerate}[{\rm(a)}]
\item\label{CORO001.3} if $T\neq\Al_m$, $\PSp_4(3)$ or $\PSL_2(p)$, then there are at least three prime divisors of $n$;
\item\label{CORO001.1} if $K$ is of Lie type over $\mathbb{F}_q$ of characteristic $p$, then $\left(|T|_p/n_p\right)^2>n$ and $|T|_p^3>|\Aut(T)|$;
\item\label{CORO001.2} if $T$ is neither $\PSL_2(7)$ nor $\Al_m$ with $m\leq10$, then $|T|>2n^2\log_2n$.
\end{enumerate}
\end{lemma}

\begin{proof}
Observe that, for even $m$, the number $(q^m-1)/2=(q^{m/2}+1)(q^{m/2}-1)/2$ has at least two prime divisors for each pair $(m,q)\notin\{(2,2),(2,3)\}$. Then statement~\eqref{CORO001.3} is easy to verify. Moreover, statements~\eqref{CORO001.1} and~\eqref{CORO001.2} for groups $T$ of Lie type are straightforward to verify. Let $T=\Al_m$, so that $n=\binom{m}{d}$ for some $d<m/2$. By the fact $\binom{m}{d}<2^m/\sqrt{m}$ from Lemma~\ref{LEM011}, we derive from Lemma~\ref{lem0.2} that
\[
2n^2\log_2n<2\cdot\frac{4^m}{m}\cdot m<2\cdot\left(\frac{m}{e}\right)^m<\frac{m!}{2}=|T|
\]
for $m\geq11$, which completes the proof.
\end{proof}

\subsection{Proof of Proposition~\ref{THM002}}\label{SUBSEC3.4}

Let $G=\Al_n$ or $\Sy_n$, and let $H\neq\Al_n$ be a maximal subgroup of $G$. By the classification of maximal subgroups of alternating and symmetric groups (see~\cite{LPS1987}, for example), one of the following holds:
\begin{enumerate}[{\rm(i)}]
\item\label{a} (intransitive) $H=(\Sy_m\times\Sy_k)\cap G$, with $n=m+k$ and $m\neq k$;
\item\label{b} (transitive imprimitive) $H=(\Sy_m\wr\Sy_k)\cap G$, with $n=mk$, $m>1$ and $k>1$;
\item\label{c} (affine) $H=\AGL_k(p)\cap G$, with $n=p^k$ and $p$ prime;
\item\label{d} (diagonal) $H=(T^k.(\Out(T)\times\Sy_k))\cap G$, with $T$ nonabelian simple, $k>1$ and $n=|T|^{k-1}$;
\item\label{e} (primitive wreath) $H=(\Sy_m\wr\Sy_k)\cap G$, with $n=m^k$, $m\geq5$ and $k>1$;
\item\label{f} (primitive almost simple) $H$ is almost simple and primitive on $\{1,2,\ldots,n\}$.
\end{enumerate}
For a maximal subgroup $H$ in a given case, our strategy for proving Proposition~\ref{THM002} is to show that every maximal subgroup with the same order as $H$ is stated as in Proposition~\ref{THM002}. This is achieved in Lemmas~\ref{PROP001} to~\ref{PROP005} for the respective cases~\eqref{a} to~\eqref{f} of $H$, while the remaining case is addressed in Proposition~\ref{THM002}.
Let us start with the first lemma.

\begin{lemma}\label{PROP001}
Let $\Al_n\leq G\leq\Sy_n$. Let $H$ be an intransitive maximal subgroup of $G$. Suppose that a maximal subgroup $K$ of $G$ has the same order as $H$. Then one of the following holds:
\begin{enumerate}[{\rm(a)}]
\item\label{PROP001.1} $H$ and $K$ are conjugate by some element in $\Al_n$;
\item\label{PROP001.2} $n=6$, and $K$ is either transitive imprimitive or $\PSL_2(5)\leq K\leq\PGL_2(5)$.
\end{enumerate}
\end{lemma}

\begin{proof}
It is straightforward to verify the conclusion when $n\leq22$. Assume $n\geq23$, and let $H=(\Sy_m\times\Sy_k)\cap G$ with $n=m+k$ and $m>k$. Note that the function $f(x)=x!\cdot(n-x)!$ is strict increasing when $x\geq n/2$. Hence, if $K$ is in intransitive, then $H$ and $K$ can be seen as the stabilizers in $G$ of some subsets of $\{1,\ldots,n\}$ with the same size, and clearly, they are conjugate in $\Al_n$. Moreover, since Lemma~\ref{lem0.2} leads to
\[
|H|\geq\frac{m!\cdot k!}{2}
\geq\frac{1}{2}\left(\left\lceil\frac{n-1}{2}\right\rceil!\right)^2
>\left(\frac{n-1}{2e}\right)^{n-1}>3^{n},
\]
it follows from Lemma~\ref{lem000} that $K$ is not primitive.

Now $K$ is transitive imprimitive, so that $K=(\Sy_{c}\wr\Sy_{n/c})\cap G$ for some proper divisors $c$ and $d$ of $n$ with $n=cd$. If $c=2$, then $|K|\leq|\Sy_2\wr\Sy_{d}|<|\Sy_{d+1}\times\Sy_{d-1}|\leq|H|$, a contradiction. If $d=2$, then the condition $|H|=|K|$ gives $\binom{n}{m}=\binom{n}{c}/2$, which implies $\binom{m}{m-c}=2\binom{c}{m-c}$, contradicting Lemma~\ref{lem0.1}. Hence, $d\notin\{2,n/2\}$. Let $r$ be the largest proper divisor of $n$ not in $\{2,n/2\}$. Then $3\leq n/r\leq r\leq\lfloor(n-1)/2\rfloor$. By Lemma~\ref{lem0.6}, we obtain $|K|\leq|(\Sy_{r}\wr\Sy_{n/r})\cap G|$ and so
\[
\frac{|H|}{|K|}
\geq\frac{\left\lfloor\frac{n+1}{2}\right\rfloor!\cdot \left\lfloor\frac{n-1}{2}\right\rfloor!}{(r!)^{\frac{n}{r}}\cdot\big(\frac{n}{r}\big)!}
=\frac{\left\lfloor\frac{n+1}{2}\right\rfloor\left(\left\lfloor\frac{n-1}{2}\right\rfloor\cdots
(r+1)\right)^2}{(r!)^{\frac{n}{r}-2}\cdot\big(\frac{n}{r}\big)!}
>\frac{\left\lfloor\frac{n+1}{2}\right\rfloor\cdot(r+1)^{n-2r-2}}{ (r^{(r-1)(\frac{n}{r}-2)-1}\cdot2)\cdot
(r^{\frac{n}{r}-3}\cdot3\cdot2)}>1,
\]
a contradiction. This completes the proof.
\end{proof}

The next lemma deals with the case where $H$ is a transitive imprimitive maximal subgroup.

\begin{lemma}\label{PROP002}
Let $\Al_n\leq G\leq\Sy_n$, and let $H$ be a transitive imprimitive maximal subgroup of $G$. Suppose that a maximal subgroup $K$ of $G$ has the same order as $H$. Then either $H$ and $K$ are conjugate by some element in $\Al_n$, or $n=6$ and $K$ is intransitive.
\end{lemma}

\begin{proof}
If $K$ is transitive imprimitive, then Lemma~\ref{lem0.7} implies that $H$ and $K$ are conjugate by some element in $\Al_n$. So, in view of Lemma~\ref{PROP001}, we only need to prove that $K$ cannot be primitive. It is straightforward to verify this when $n\leq48$. Assume $n\geq49$, and let $H=(\Sy_m\wr\Sy_k)\cap G$ with $n=mk$. Then Lemma~\ref{lem0.2} implies that
\[
|H|\geq\frac{(m!)^k\cdot k!}{2}>\max\left\{2^k,\left(\frac{m}{e}\right)^{n}\right\}\cdot \left(\frac{k}{e}\right)^k>2^n.
\]
We thus conclude from Lemma~\ref{lem000} that $K$ is not primitive, as desired.
\end{proof}

For a positive integer $n$ and a prime $p$, recall that $n_p$ is the largest power of $p$ dividing $n$ and that $v_p(n)=\log_p(n_p)$ is the exponent of $p$ in $n_p$.

\begin{lemma}\label{PROP003}
Let $\Al_n\leq G\leq\Sy_n$, and let $H$ be an affine maximal subgroup of $G$. Suppose that a maximal subgroup $K$ of $G$ has the same order as $H$. Then $H$ and $K$ are conjugate in $\Sy_n$.
\end{lemma}

\begin{proof}
Let $H=\AGL_k(p)\cap G$ with $n=p^k$. The conclusion is clearly true if $K$ is also affine. By Lemmas~\ref{PROP001} and~\ref{PROP002}, it suffices to show that, assuming that $K$ is primitive, $K$ must be affine. In fact, considering the value of $n$, we see that $K$ cannot be diagonal.

Assume that $K=(\Sy_{p^a}\wr\Sy_{b})\cap G$ with $n=p^{ab}$, where $p^a\geq5$ and $b>1$. Define $d=1$ if $p=2$ and $|H|=|\AGL_k(p)|/2$, and $d=0$ otherwise. Note that
\[
|H|_p^2=p^{k(k+1)-2d}>p^{\frac{k(k+1)-d}{2}}\cdot\prod_{i=1}^{k}(p^i-1)\geq|H|.
\]
Hence, $|K|<|K|_p^2$, and we derive from Lemmas~\ref{lem0.2} and~\ref{lem0.5}\eqref{lem0.5.1} that
\[
\left(\frac{p^a}{e}\right)^{p^ab}\cdot b!<\frac{1}{2}\left(p^a!\right)^b\cdot b!\leq|K|<|K|_p^2\leq p^{\frac{2(p^a-1)b}{p-1}}\cdot p^{\frac{2(b-1)}{p-1}}<\left(p^{\frac{2}{p-1}}\right)^{p^ab}.
\]
This yields $p^a\leq2^3$ and $b\leq3$. For these cases, one may directly check that $|H|\neq|K|$, a contradiction.

For the rest of the proof, assume that $K$ is primitive almost simple with socle $T$. Then $T$ is a transitive subgroup of $\Sy_n$, and the stabilizer $L$ of $T$ has index $n=p^a$. All possibilities for such $(T,L)$ are listed in Lemma~\ref{lem0.0}. It is straightforward to verify that $|H|\neq|K|$ when $T$ is in~\eqref{lem0.0.1} or~\eqref{lem0.0.3} of Lemma~\ref{lem0.0}. Thus, suppose that $T=\PSL_m(q)$ and $(q^m-1)/(q-1)=n=p^k$. Then, it is clear that $(q,m)\neq(2,6)$. Moreover, if $(q,m)=(2^r-1,2)$ for some $r$, then $p^k=q+1=2^r$, $v_2(|T|)=r=k$ and so
\[
\frac{k(k+1)}{2}-1\leq v_2(|H|)=v_2(|K|)\leq v_2(|T|)+v_2(|\Out(T)|)<k+\log_2(2k).
\]
This inequality holds only when $k\leq3$, and straightforward computation then shows $|H|\neq|K|$. Consequently, $(q,m)\neq(2,6)$ or $(2^r-1,2)$ for any $r$. By Lemma~\ref{lem0.10}, there exists a prime which divides $q^m-1$ but not $q^i-1$ for any positive integer $i<m$. In fact, this prime is $p$ as $(q^m-1)/(q-1)=p^k$. Note that $v_p(|T|)=k$ and that
\[
v_p(|\Out(T)|)\leq\log_p(\log_2q)<\log_p(k\log_2p).
\]
We obtain
\[
\frac{k(k+1)}{2}\leq v_p(|H|)=v_p(|K|)<k+\log_p(k\log_2p),
\]
which implies that one of the following holds:
\begin{itemize}
\item $(k,p)=(3,2)$,
\item $k=2$ and $p\leq3$,
\item $k=1$.
\end{itemize}
By a direct calculation, the first two cases contradict $|H|=|K|$. Suppose that $k=1$. Then
\[
|\PSL_m(q)|\leq|K|=|H|\leq|\AGL_1(p)|<p^2=\left(\frac{q^m-1}{q-1}\right)^2.
\]
This leads to $m=2$, and so $p-1=q$ is a prime power. However, $|K|=|H|$ divides $|\AGL_1(p)|=p(p-1)=pq$, which is impossible as $K$ is nonsolvable. The proof is complete.
\end{proof}

The following lemma addresses the case where $H$ is a diagonal maximal subgroup.

\begin{lemma}\label{PROP004}
Let $\Al_n\leq G\leq\Sy_n$, and let $H$ be a diagonal maximal subgroup of $G$. Suppose that a maximal subgroup $K$ of $G$ has the same order as $H$. Then $K$ is also diagonal. Moreover, $H$ and $K$ have socle $S^k$ and $T^k$ respectively, where $S\cong T$ is nonabelian simple or $\{S,T\}=\{\POm_{2m+1}(q),\PSp_{2m}(q)\}$ with $q$ odd and $m\geq3$.
\end{lemma}

\begin{proof}
Let $H=(S^k.(\Out(S)\times\Sy_k))\cap G$ with $S$ nonabelian simple. In light of Lemmas~\ref{PROP001}--\ref{PROP003}, we only need to consider that $K$ is diagonal, primitive wreath, or primitive almost simple. Assume that $K$ is diagonal with socle $T^\ell$, where $T$ is nonabelian simple. Then $n=|S|^{k-1}=|T|^{\ell-1}$. It follows from~\cite[Theorem~6.1]{KLST1990} that $k=\ell$ and either $S\cong T$ or
\[
\{S,T\}\in\big\{\{\POm_{2m+1}(q),\PSp_{2m}(q)\}\mid\text{$q$ odd and $m\geq3$}\big\}\cup\{\PSL_4(2),\PSL_3(4)\}.
\]
All of these cases can occur except for $\{S,T\}=\{\PSL_4(2),\PSL_3(4)\}$, which is excluded by checking $|H|$ and $|K|$, as desired.

Assume that $K=(\Sy_a\wr\Sy_b)\cap G$, where $a\geq5$ and $b\geq2$. Then since $a^b=n=|S|^{k-1}$ has at least three prime divisors, it holds $a\geq30$. Applying Lemma~\ref{lem0.2}, we obtain
\[
|S|^{k+1}\cdot k^k>|H|=|K|\geq\left(a!\right)^b>\left(\frac{a}{e}\right)^{ab}>a^{7b}
=n^7=|S|^{7(k-1)}.
\]
This yields that $k^k>|S|^{6k-8}\geq|S|^{2k}$ and so $k>|S|^{2}>2|S|$. By Lemma~\ref{lem0.4}, there exists a prime $p\in(k/2,k)$, so that $v_p(|K|)=v_p(|H|)=1$, and so we have $b\geq p>k/2$. Moreover, $b<2k$ (otherwise, Lemma~\ref{lem0.4} indicates the existence of a prime in $(k,b)$ dividing $|K|$ but not $|H|$), and it follows that
\[
a=|S|^{\frac{k-1}{b}}>|S|^{\frac{k-1}{2k}}\geq|S|^{\frac{1}{4}}>2|\Out(S)|,
\]
where the last inequality is from Lemma~\ref{LEM008} together with
direct verification for small $S$. Take a prime $r\in(a/2,a)$, and note that $\gcd(r,|S|)=1$ as $\gcd(r,|S|^{(k-1)/b})=\gcd(r,a)=1$. We conclude that $b+v_r(b!)=v_r(|K|)=v_r(|H|)=v_r(k!)$, which contradicts Lemma~\ref{lem0.5}\eqref{lem0.5.2} as $k<2b$.

Assume that $K$ is primitive almost simple with socle $T$, and let $L$ be a stabilizer in $K$ for the primitive action on $\{1,2,\ldots,n\}$. Since $H=(S^k.(\Out(S)\times\Sy_k))\cap G$ and $n=|S|^{k-1}$, every prime divisor of $|H|$ divides $|H|/n$. It follows that every prime divisor of $|K|=|H|$ divides $|H|/n=|K|/n=|L|$, and so all possibilities for such $(T,L)$ are listed in~\cite[Table~10.7]{LPS2000}, as Lemma~\ref{lem0.8} asserts. A direct inspection shows that the finite cases in~\cite[Table~10.7]{LPS2000} cannot happen, and we only need to consider the infinite families of triples $(K,T,L)$ in Table~\ref{TABLE001}.

First, let $T=\Al_m$ and $L\cap T=(\Sy_m\times\Sy_{m-d})\cap\Al_m$, as in the first row of Table~\ref{TABLE001}. For $m\leq10$, one may directly verify that $|H|\neq|K|$. Therefore, $m\geq11$. If $k=2$, then
\[|T|\leq|H|\leq2|S|^2|\Out(S)|\leq2n^2\log_2n,
\]
contradicting Lemma~\ref{CORO001}\eqref{CORO001.2}. Thus, let $k\geq3$. As $n=|S|^{k-1}=\binom{m}{d}$, it follows from~\cite[Chapter~3]{TheBook} that $d\leq4$ and so $n\leq\binom{m}{4}<(m/2)^4$. By Lemma~\ref{lem0.4}, there exists a prime $r\in(m/2,m)$. Then
\begin{equation}\label{eq0.4}
r>\frac{m}{2}>n^{\frac{1}{4}}=|S|^{\frac{k-1}{4}}\geq
|S|^{\frac{1}{2}}>|\Out(S)|.
\end{equation}
Moreover, take a prime divisor $p\geq5$ of $|S|$. Then Lemma~\ref{lem0.5}\eqref{lem0.5.1} yields that
\[
r>\frac{m}{2}>\frac{m-1}{p-1}\geq v_p(|K|)=v_p(|H|)\geq k.
\]
Combining this with~\eqref{eq0.4}, we conclude that $v_r(|H|)=kv_r(|S|)\neq1=v_r(|K|)$, a contradiction.

Next let $K$ be a simple group of Lie type over $\mathbb{F}_q$ of characteristic $p$ as described in Table~\ref{TABLE001}. Then we derive from Lemma~\ref{CORO001}\eqref{CORO001.1} and Lemma~\ref{lem0.5}\eqref{lem0.5.1} that
\[
|S|^{k-1}=n<\left(\frac{|T|}{n}\right)_p^2\leq
\left(\frac{|H|}{n}\right)_p^2
\leq|S|_p^2\cdot|\Out(S)|^2\cdot\left(p^{\frac{2}{p-1}}\right)^{k-1}
\leq|S|_p^2\cdot|S|\cdot4^{k-1}.
\]
If $k\geq4$, then $(|S|/4)^3\leq(|S|/4)^{k-1}<|S|_p^2\cdot|S|$ and so $|S|^2<64|S|_p^2$, impossible. Therefore, $k\leq3$, and we derive from Lemma~\ref{LEM008} that
\[
|T|\leq|K|=|H|\leq|S|^{k}\cdot|\Out(S)|\cdot k!<2|S|^{2(k-1)}\cdot\log_2|S|\leq2n^2\log_2n,
\]
which contradicts Lemma~\ref{CORO001}\eqref{CORO001.2} as it is straightforward to verify that $T\neq\PSL_2(7)$. This completes the proof.
\end{proof}

We now deal with the case where $H$ is a primitive wreath maximal subgroup.

\begin{lemma}\label{PROP005}
Let $\Al_n\leq G\leq\Sy_n$, and let $H$ be a primitive wreath maximal subgroup of $G$. Suppose that a maximal subgroup $K$ of $G$ has the same order as $H$. Then $H$ and $K$ are conjugate in $\Sy_n$.
\end{lemma}

\begin{proof}
In view of Lemmas~\ref{PROP001}--\ref{PROP004}, it suffices to consider that $K$ is primitive wreath or primitive almost simple. In fact, if $K$ is primitive wreath, the conclusion immediately follows from Lemma~\ref{lem0.9}. In the following, suppose that $K$ is primitive almost simple. Let $T=\Soc(K)$, let $L$ be a stabilizer in $K$ for the primitive action on $\{1,2,\ldots,n\}$, and let $H=(\Sy_m\wr\Sy_k)\cap G$ with $n=m^k$, where $m\geq5$ and $k>1$.

Suppose that $m$ is not a prime. Since $H=(\Sy_m\wr\Sy_k)\cap G$ and $n=m^k$, every prime divisor of $|H|$ divides $|H|/n$. It follows that every prime divisor of $|K|=|H|$ divides $|H|/n=|K|/n=|L|$, and so all possibilities for such $(T,L)$ are listed in~\cite[Table~10.7]{LPS2000}, as Lemma~\ref{lem0.8} asserts. A direct inspection shows that the finite cases cannot happen, and we only need to consider the infinite families of triples $(K,T,L)$ as listed in Table~\ref{TABLE001}. If $T=\Al_c$ for some $c$, then by Lemma~\ref{lem0.4}, there exists a prime $r\in\left(c/2,c\right)$, so that $v_r(|H|)=v_r(|K|)=1$, which implies that $k\geq r>c/2$. However, this contradicts the fact from Lemma~\ref{lem0.5}\eqref{lem0.5.1} that $k\leq v_5(|H|)=v_5(|K|)\leq(c-1)/4<c/2$. Consequently, $K$ is of Lie type in Table~\ref{TABLE001}. Note that $m^k=n=|T\,{:}\,L\cap T|$ is listed in Table~\ref{TABLE001}. If $T=\PSL_2(p)$, where $p$ is a Mersenne prime, then $m^k=n=p+1$ must be a power of $2$ to a prime exponent, contradicting $m\geq5$ and $k>1$. Moreover, one may directly verify that $T\neq\PSp_4(3)$. Hence, we deduce from Lemma~\ref{CORO001}\eqref{CORO001.3} that $n=m^k$ has at least three prime divisors and so $m\geq30$. However, this together with Lemmas~\ref{lem0.5} and~\ref{lem0.2} yields that for the characteristic $p$,
\[
|T|_p\leq|H|_p\leq\left(p^{\frac{m-1}{p-1}}\right)^k\cdot p^{\frac{k-1}{p-1}}=\left(p^{\frac{1}{p-1}}\right)^{mk-1}
\leq2^{mk-1}<\left(\frac{m}{e}\right)^{\frac{mk}{3}}
<(m!)^{\frac{k}{3}}<|H|^{\frac{1}{3}}=|K|^{\frac{1}{3}},
\]
which contradicts the fact $|T|_p^3>|\Aut(T)|$ from Lemma~\ref{CORO001}\eqref{CORO001.1}.

To complete the proof, let $m$ be a prime. Then $T$ is a transitive subgroup of $\Sy_n$ with a stabilizer $L$ of prime power index. All possibilities for such $(T,L)$ are listed in Lemma~\ref{lem0.0}. It is straightforward to check that $|H|\neq|K|$ if $K$ is in~\eqref{lem0.0.1} or~\eqref{lem0.0.3} of Lemma~\ref{lem0.0}. So let $T=\PSL_a(q)$ with $q=p^f$, and without loss of generality, let $(p,af)\neq(2,6)$. Now $n=(p^{af}-1)/(p^f-1)=m^k$, where $af>2$ as $m$ is an odd prime. In view of Lemma~\ref{lem0.10}, we see that $m$ is a primitive prime divisor of $(p,af)$. As a consequence, $m\equiv1\ (\bmod\ af)$, and in particular, $m>af>f$. Hence,
\[
|K|\leq|\Aut(T)|<2f\cdot q^{a^2-1}<2m\cdot q^{\frac{3}{2}a(a-1)}=2m|T|_p^3.
\]
Combining this with $|H|=|K|$ and Lemmas~\ref{lem0.2} and~\ref{lem0.5}\eqref{lem0.5.1}, we obtain
\[
\left(\frac{m}{e}\right)^{mk}<\frac{(m!)^k}{2m}\leq\frac{|H|}{2m}<|T|_p^3\leq|H|_p^3
\leq\left(p^{\frac{m-1}{p-1}}\right)^{3k}\cdot\left(p^{\frac{k-1}{p-1}}\right)^3<\left(p^{\frac{3}{p-1}}\right)^{mk}.
\]
This holds only in the following cases:
\begin{itemize}
\item $p=2$, $m\leq21$,
\item $p=3$, $m\leq14$,
\item $p=5$, $m\leq9$,
\item $p=7$, $m\leq7$,
\item $p=11$, $m=5$.
\end{itemize}
Recall that $m\geq5$ is a prime, $k>1$ and $af\mid(m-1)$. Then it is readily verified that the equation $(p^{fa}-1)/(p^f-1)=m^k$ holds only when $(p,a,f,m)=(3,5,1,11)$. However, $T\neq\PSL_5(3)$ as $|H|=|K|$. This completes the proof.
\end{proof}

We are now ready to prove Proposition~\ref{THM002}.

\begin{proof}[Proof of Proposition~$\ref{THM002}$]
For $n=6$, the conclusion can be verified by \textsc{Magma} computation. Thus, we may assume without loss of generality that $\Al_n\leq G\leq\Sy_n$. Then the conclusion follows from Lemmas~\ref{PROP001}--\ref{PROP005} and Proposition~\ref{THM004}.
\end{proof}

\section{Biprimitive semisymmetric graphs of alternating and symmetric groups}\label{SEC4}

In this section, we investigate biprimitive semisymmetric graphs of alternating groups $\Al_n$ and symmetric groups $\Sy_n$. A general construction of semisymmetric graphs of $\Al_n$ is presented in the first subsection. We then use this construction, along with previously established results, to show a classification of biprimitive semisymmetric graphs of $\Al_n$ and $\Sy_n$, namely, Theorem~\ref{THM003}. In Subsection~\ref{SUBSEC4.3}, for such graphs of twice odd order, a more explicit characterization is given in Proposition~\ref{CORO003}. At the end of this subsection, we exhibit an example of semisymmetric graphs in case~\eqref{CORO003.2} but not in case~\eqref{CORO003.1} of Proposition~\ref{CORO003}, which also serves as an example for case~(\ref{THM003.1}.3) of Theorem~\ref{THM003}.

\subsection{Semisymmetric graphs of alternating groups and proof of Theorem~\ref{THM003}}\label{SUBSEC4.1}

The following proposition provides an equivalent group-theoretic characterization for a large family of biprimitive graphs $\B(\Al_n,H,H,HgH)$ to be semisymmetric.

\begin{proposition}\label{CON001}
Let $G=\Al_n$ with $n>6$, and let $H$ be a maximal subgroup of $G$ with index $d$. Suppose that each overgroup of the primitive group $G$ acting on $[G\,{:}\,H]$ by right multiplication has socle $G$ or $\Al_d$. Then for an element $g\in G$ with $HgH\neq Hg^{-1}H$, the bi-coset graph $\B(G,H,H,HgH)$ is semisymmetric if and only if either $\N_{\Sy_n}(H)=H$, or $\N_{\Sy_n}(H)>H$ and $Hg^{-1}H\neq Hg^{\iota}H$ for all (or equivalently, for some) $\iota\in\N_{\Sy_n}(H)\setminus H$.
\end{proposition}

\begin{proof}
Let $g\in G$ with $HgH\neq Hg^{-1}H$, and let $N=\N_{\Sy_n}(H)$. Then $|N\,{:}\,H|=1$ or $2$. Let $\Gamma=\B(G,H,H,HgH)$ with the left part $U$ and right part $W$. Let $A^+$ be the subgroup of $\Aut(\Gamma)$ fixing setwise $U$ and $W$. To avoid confusion, we write the vertices from $U$ as $Hx$ and those from $W$ as $Kx$, where $K=H$ and $x\in G$. We first show the ``only if'' part. Suppose on the contrary that $N>H$ and there is an element $\iota\in N\setminus H$ such that $Hg^{-1}H=Hg^{\iota}H$. Let $\varphi$ be a permutation on $U\cup W$ which maps $Hx$ to $Kx^{\iota}$ and maps $Kx$ to $Hx^{\iota}$ for each $x\in G$. It is straightforward to verify that $\varphi$ is well-defined and bijective. Moreover, noting that $H\iota=H\iota^{-1}$ as $G$ has index $2$ in $\Sy_n$, we obtain
\begin{align*}
\{Hx,Ky\}\in E(\Gamma)\Leftrightarrow yx^{-1}\in HgH&\Leftrightarrow xy^{-1}\in Hg^{-1}H=Hg^{\iota}H\\
&\Leftrightarrow x^{\iota}(y^{\iota})^{-1}\in HgH\Leftrightarrow \{Hy^{\iota},Kx^{\iota}\}\in E(\Gamma).
\end{align*}
This means that $\varphi\in\Aut(\Gamma)\setminus A^+$, contradicting the assumption that $\Gamma$ is semisymmetric.

Next we prove the ``if'' part. Since $HgH\neq Hg^{-1}H$ implies that $\Gamma$ is not a complete bipartite graph, a perfect matching, or their bipartite complements. We derive from Lemmas~\ref{LEM005} and~\ref{CORO002} that $A^+$ acts faithfully on both $U$ and $W$ with $\Soc(A^+)\neq\Al_d$. Thus, by the condition of the proposition, it holds $\Soc(A^+)=G$. Suppose on the contrary that $\Gamma$ is vertex-transitive, namely, $\Aut(\Gamma)=A^+.2$. Since $HgH\neq Hg^{-1}H$, it follows from Lemmas~\ref{LEM004} and~\ref{LEM002} that $\Aut(\Gamma)$ also has socle $G$. In particular, $\Aut(\Gamma)=\Sy_n$ and $A^+=G$. Since $H$ fixes exactly two vertices of $\Gamma$, the subgroup of $\Sy_n$ stabilizing the set of these two vertices normalizes $H$, and thus, $N>H$. Let $\iota\in N\setminus H$. Now the action of $\Sy_n$ on the vertex set of $\Gamma$ is permutation-isomorphic to the right multiplication action of $\Sy_n$ on $[\Sy_n\,{:}\,L]$ such that the following diagram
\vspace{-0.5em}
\begin{figure}[H]
\centering
\begin{tikzpicture}[scale=0.8]
\node at (0, 0) {$[\Sy_n\,{:}\,L]$};
\node at (5.5, 0) {$[\Sy_n\,{:}\,L]$};
\node at (0, 2) {$U\cup W$};
\node at (5.5, 2) {$U\cup W$};
\draw[->, >=latex, line width=0.25mm] (0.9,0) -- (4.6,0);
\draw[->, >=latex, line width=0.25mm] (0.9,2) -- (4.6,2);
\draw[->, >=latex, line width=0.25mm] (-0.06,1.65) -- (-0.06,0.35);
\draw[->, >=latex, line width=0.25mm] (5.44,1.65) -- (5.44,0.35);
\node at (-0.5, 1.1) {$\psi$};
\node at (5.9, 1.1) {$\psi$};
\node at (2.7, 0.4) {$R(\Sy_n)$};
\node at (2.4, 2.45) {$G\leq\Sy_n$};
\draw[->, >=latex, line width=0.25mm] (2.8,2.15) -- (2.8,0.75);
\node at (3.3, 1.4) {$R$};
\end{tikzpicture}
\end{figure}
\vspace{-0.5em}
\noindent commutes, where $\psi$ is a bijection, $L=H$ (we introduce $L$ to avoid confusion between points in $U\cup W$ and $[\Sy_n\,{:}\,H]$),  and $R$ is the right multiplication action. Since $H$ and $K$ are the only two vertices of $\Gamma$ stabilized by $H$ while $L$ and $L\iota$ are the only two points in $[\Sy_n\,{:}\,L]$ stabilized by $R(L)$, it follows that $H^\psi=L$ and $K^\psi=L\iota$. Hence, for each $x\in G$,
\[
(Hx)^\psi=H^{R(x)\psi}=H^{\psi R(x)}=Lx\ \text{ and }\
(Kx)^\psi=K^{R(x)\psi}=K^{\psi R(x)}=L\iota x.
\]
Take an arbitrary $y\in G$. The images of $H$ and $Ky$ under $\iota\in\Sy_n$ are
\[
H^\iota=H^{\psi R(\iota)\psi^{-1}}=(L\iota)^{\psi^{-1}}=K\ \text{ and }\
(Ky)^\iota=(Ky)^{\psi R(\iota)\psi^{-1}}=(L\iota y\iota)^{\psi^{-1}}=H\iota^{-1}y\iota,
\]
where the last equation follows from $L\iota=L\iota^{-1}$ and $\iota^{-1}y\iota\in G$. This yields that
\begin{align*}
y\in HgH\Leftrightarrow\{H,Ky\}\in E(\Gamma)
&\Leftrightarrow\{(Ky)^\iota,H^\iota\}\in E(\Gamma)\\
&\Leftrightarrow\{H\iota^{-1}y\iota,K\}\in E(\Gamma)\Leftrightarrow\iota^{-1}y^{-1}\iota\in HgH
\Leftrightarrow\iota^{-1}y\iota\in Hg^{-1}H.
\end{align*}
We thereby conclude that
\[
y\in Hg^{\iota}H
\Leftrightarrow\iota y\iota^{-1}\in HgH
\Leftrightarrow\iota^{-1}(\iota y\iota^{-1})\iota\in Hg^{-1}H
\Leftrightarrow y\in Hg^{-1}H,
\]
contradicting the assumption that $Hg^{-1}H\neq Hg^{\iota}H$. This completes the proof.
\end{proof}

\begin{remark}\label{RMK003}
In the proof of Proposition~\ref{CON001}, it can be observed that the proposition still holds if the supposition that ``each overgroup of the primitive group $G$ acting on $[G\,{:}\,H]$ by right multiplication has socle $G$ or $\Al_d$'' is replaced by ``$\Soc(A^+)=\Al_n$ or $\Al_d$'', where $A^+$ is the automorphism group of $\B(G,H,H,HgH)$ that preserves each part setwise.
\end{remark}

Proposition~\ref{CON001} (under its assumptions) gives a criterion for determining when a biprimitive graph of $\Al_n$ is semisymmetric. This is a relaxation of the criterion Theorem~\ref{THM001}\eqref{THM001.3} for biprimitive graphs of $\Al_n$ to be semisymmetric. In fact, if $K=\N_{\Sy_n}(H)>H$, then for $g\in G$ and $\iota\in K\setminus H$, we have
\[
\Al_n(K\cap K^g)=\Sy_n
\,\Rightarrow\,
H\iota^{-1}\cap(H\iota^{-1})^g\neq\emptyset
\,\Rightarrow\,
gH\cap Hg^{\iota}\neq\emptyset
\,\Rightarrow\,
Hg^{\iota}H=HgH\neq Hg^{-1}H,
\]
as stated in Proposition~\ref{CON001}.

We are now in a position to prove Theorem~\ref{THM003}.

\begin{proof}[Proof of Theorem~$\ref{THM003}$]
Let $G$ have socle $\Al_n$. Then Lemma~\ref{LEM003} (together with the analysis in the paragraph following it) implies that $\Gamma$ is isomorphic to the bi-coset graph $\B(G,L,R,RgL)$ for some $g\in G$ and some maximal subgroups $L$ and $R$ of $G$. View $\Gamma$ as such a bi-coset graph, and let $U$ and $W$ be the biparts of $\Gamma$. Since the smallest biprimitive semisymmetric graph has $80$ vertices~\cite{DM1999}, we have $d:=|U|\geq40$, and so $n>6$. In addition, since $\Gamma$ is semisymmetric, it is neither a complete bipartite graph nor an empty graph. Then Lemma~\ref{LEM005} states that $G$ acts faithfully on both $U$ and $W$.

Assume that $L$ and $R$ are conjugate in $G$. By Lemma~\ref{LEM003} and the semisymmetry of $\Gamma$, we may let $L=R$ and $RgL\neq Rg^{-1}L$. If there exists an overgroup $H$ of $G$ in $\Sy_d$ whose socle is neither $T$ nor $\Al_d$, then (\ref{THM003.1}.1) of Theorem~\ref{THM003} follows by~\cite[Table~II--VI]{LPS1987}. Now assume that, except for $\Al_d$ and $\Sy_d$, each overgroup of $G$ on $U$ has the same socle as $G$. If $G=\Sy_n$, then the triple $(G,L,g)$ satisfies conditions of Theorem~\ref{THM001}, and so $\Gamma$ is semisymmetric. If $G=\Al_n$, then $\Gamma$ is a semisymmetric graph described in Proposition~\ref{CON001}, as stated in~(\ref{THM003.1}.3) of Theorem~\ref{THM003}.

Assume that $L$ and $R$ are not conjugate in $G$. Since $L$ and $R$ have the same order, all the possibilities for $(L,R)$ are outlined in Proposition~\ref{THM002}. If $G=\Sy_n$ and $L$ and $R$ are isomorphic and non-conjugate, then inspecting~\eqref{THM002.2}--\eqref{THM002.3} of Proposition~\ref{THM002}, we see that $L$ and $R$ are almost simple with the same socle, as in Theorem~\ref{THM003}\eqref{THM003.2}. If $G=\Sy_n$ and $L$ and $R$ are non-isomorphic, then $L$ and $R$ are described in Proposition~\ref{THM002}\eqref{THM002.2}--\eqref{THM002.3}, leading to Theorem~\ref{THM003}\eqref{THM003.3}. Finally, if $G=\Al_n$, then either $L$ and $R$ are isomorphic and non-conjugate in $G$ as in Theorem~\ref{THM003}\eqref{THM003.2}, or $L$ and $R$ are non-isomorphic, which lies in Proposition~\ref{THM002}\eqref{THM002.2}--\eqref{THM002.3} and hence in Theorem~\ref{THM003}\eqref{THM003.3}.
\end{proof}

\subsection{Proof of Proposition~\ref{CORO003} and an example}\label{SUBSEC4.3}

Based on previously established results, we can obtain Proposition~\ref{CORO003} with minimal in addition effort.

\begin{proof}[Proof of Proposition~$\ref{CORO003}$]
Let $U$ and $W$ be the biparts of $\Gamma$ with $|U|=|W|=d$, and let $A^+$ be the subgroup of $\Aut(\Gamma)$ fixing setwise $U$ and $W$. Then it follows from Lemma~\ref{LEM005} that $A^+$ acts faithfully on both $U$ and $W$. In view of Lemma~\ref{LEM003}, $\Gamma$ is isomorphic to the bi-coset graph $\B(G,L,R,RgL)$ for some $g\in G$ and some core-free maximal subgroups $L$ and $R$ of $G$. Since $|G\,{:}\,L|=|G\,{:}\,R|=d$ is odd and $n\geq8$, we obtain by the classification of primitive groups of odd degree~\cite[Theorem]{LS1985} (see also~\cite{Kantor1987}) that $L$ and $R$ are either an intransitive subgroup or an imprimitive subgroup. Combining this with Lemmas~\ref{PROP001} and~\ref{PROP002}, we deduce that $L$ and $R$ are $G$-conjugate. Hence, it follows from Lemma~\ref{LEM003}\eqref{LEM003.2} that $\Gamma\cong\B(G,H,H,HgH)$ for some intransitive or imprimitive subgroup $H$. Note that Theorem~\ref{THM001}\eqref{THM001.3} always holds when $G=\Sy_n$ or $H$ is an intransitive subgroup of $G$. Therefore, either $(G,H,g)$ satisfies Theorem~$\ref{THM001}$\eqref{THM001.3}, or $H$ is an imprimitive subgroup of $G=\Al_n$, as desired. To complete the proof, it remains to prove the necessary and sufficient conditions stated in Proposition~\ref{CORO003} for the semisymmetry of $\Gamma$.

First assume that each overgroup of $G$ in $\Sy_d$ has socle $\Al_n$ or $\Al_d$. If $(G,H,g)$ satisfies Theorem~\ref{THM001}\eqref{THM001.3}, then Theorem~\ref{THM001} implies that $\Gamma$ is semisymmetric if and only if $HgH\neq Hg^{-1}H$, as in Proposition~\ref{CORO003}\eqref{CORO003.1}. If $H$ is an imprimitive subgroup of $G=\Al_n$, then $\N_{\Sy_n}(H)>H$ and it follows from Proposition~\ref{CON001} that $\Gamma$ is semisymmetric if and only if $Hg^{-1}H\neq Hg^{\iota}H$ for every $\iota\in\N_{\Sy_n}(H)\setminus H$, as Proposition~\ref{CORO003}\eqref{CORO003.2} asserts.

Next assume that there exists an overgroup of $G$ in $\Sy_d$ whose socle is neither $\Al_n$ nor $\Al_d$. Inspecting~\cite[Table~II--VI]{LPS1987}, we see that $d=495$ and the minimal overgroup of $G$ is the permutation group $M:=\Om_{10}^-(2)$ on the set of singular $1$-spaces of the corresponding orthogonal space. Suppose that $M\leq A^+$. Then $\Gamma$ is $M$-biprimitive, and Lemma~\ref{LEM005} asserts that $M$ acts faithfully on both $U$ and $W$. Since the permutation representation of $M$ on the singular $1$-spaces is unique up to $M$-conjugate, the vertex stabilizers of $U$ and $W$ in $M$ must be $M$-conjugate. Therefore, by Lemma~\ref{LEM003}, the valency of $\Gamma$ is equal to some subdegree of $M$, and so $G$ and $M$ share a common suborbit, which contradicts the result in~\cite[Theorem~1]{LPS2002}. Hence, $\Soc(A^+)=\Al_n$ or $\Al_d$. If $(G,H,g)$ satisfies Theorem~\ref{THM001}\eqref{THM001.3}, then we derive from Theorem~\ref{THM001} and Remark~\ref{RMK002} that $\Gamma$ is semisymmetric if and only if $HgH\neq Hg^{-1}H$, as Proposition~\ref{CORO003}\eqref{CORO003.1} asserts. If $H$ is an imprimitive subgroup of $G=\Al_n$, then $\N_{\Sy_n}(H)>H$ and so Proposition~\ref{CON001} and Remark~\ref{RMK003} imply that $\Gamma$ is semisymmetric if and only if $Hg^{-1}H\neq Hg^{\iota}H$ for every $\iota\in\N_{\Sy_n}(H)\setminus H$, as in Proposition~\ref{CORO003}\eqref{CORO003.2}. This completes the proof.
\end{proof}

For semisymmetric graphs in Proposition~\ref{CORO003}\eqref{CORO003.1}, one may construct numerous such examples by using a method similar to the constructions of the two examples in Subsection~\ref{SUBSEC2.3}. In the following, we give a biprimitive semisymmetric graph in case~\eqref{CORO003.2} of Proposition~\ref{CORO003}, but not in case~\eqref{CORO003.1}.

\begin{example}\label{EX1}
Let $G=\Al_{32}$ be the alternating group, let $v=\{\{4i+1,4i+2,4i+3,4i+4\}\mid0\leq i\leq7\}$, and let $H$ be the stabilizer of $v$ in $G$. Take $g$ to be the element of $G$:
\[
(2,5)(3,9)(4,13)(7,10,15,18,11,17,8,14)(12,29,20,25,19,22,16,21)(23,26)(24,30)(28,31).
\]
Then $\B(G,H,H,HgH)$ is a semisymmetric graph of twice odd order in case~\eqref{THM003.2} of Proposition~\ref{CORO003}, but not in case~\eqref{CORO003.1}.
\end{example}

\begin{proof}[Proof of Example~$\ref{EX1}$]
Let $M$ be the stabilizer of $v$ in $\Sy_{32}$. Computation in \textsc{Magma}~\cite{BCP1997} shows that $M\cap M^g\leq G$. Thus, $(G,H,g)$ does not satisfy Theorem~\ref{THM001}\eqref{THM001.3}, and so $\B(G,H,H,HgH)$ is not in case~\eqref{CORO003.1} of Proposition~\ref{CORO003}. Note that $|G\,{:}\,H|=|\Sy_{32}|/|\Sy_4\wr\Sy_8|=59287247761257140625$. According to~\cite[Table~II--VI]{LPS1987}, each overgroup of the primitive group $G$ acting on $[G\,{:}\,H]$ by the right multiplication has socle $G$ or $\Al_{|G\,{:}\,H|}$. Let $\iota=(1,2)\in M\setminus H$, $x=g^{-1}$ and $y=g^\iota$. To prove that $\B(G,H,H,HgH)$ is semisymmetric, by Proposition~\ref{CON001}, we only need to prove $HgH\neq HxH$ and $HxH\neq HyH$. In fact, by direct computation, we obtain the intersection matrices (recall the definition in Subsection~\ref{SUBSEC2.3})
\[
P:=P(v,v^g)=P(v,v^y)=
\begin{pmatrix}
1 & 1 & 1 & 1 & 0 & 0 & 0 & 0 \\
1 & 1 & 0 & 1 & 1 & 0 & 0 & 0 \\
1 & 1 & 0 & 0 & 1 & 1 & 0 & 0 \\
1 & 1 & 1 & 0 & 0 & 1 & 0 & 0 \\
0 & 0 & 1 & 1 & 0 & 0 & 1 & 1 \\
0 & 0 & 0 & 1 & 1 & 0 & 1 & 1 \\
0 & 0 & 0 & 0 & 1 & 1 & 1 & 1 \\
0 & 0 & 1 & 0 & 0 & 1 & 1 & 1 \\
\end{pmatrix}
\]
and $Q:=P(v,v^x)=P^\mathsf{T}$ is the transpose of $P$. Observe that there exists two column vectors of $P$ that have exactly $4$ common coordinates with entry $1$. This also applies to $XPY$ for any permutation matrices $X$ and $Y$, but does not hold for the matrix $P^\mathsf{T}=Q$. Consequently, there are no permutation matrices $X$ and $Y$ such that $Q=XPY$. It follows from Lemma~\ref{LEM012} that $x\notin HgH\cup HyH$, and so $HgH\neq HxH$ and $HxH\neq HyH$, as desired.
\end{proof}

\section*{Acknowledgments}

The second author was partially supported by NNSFC (12471022).


\begin{thebibliography}{}

\bibitem{TheBook}
M. Aigner and G.~M. Ziegler, \emph{Proofs from the Book}, Springer, Berlin, 2018.

\bibitem{Artin1955}
E. Artin, The orders of the classical simple groups, \emph{Comm. Pure Appl. Math.} 8 (1955), 455--472.

\bibitem{BCP1997}
W. Bosma, J. Cannon and C. Playoust, The magma algebra system I: The user language, \emph{J. Symbolic Comput.} 24 (1997), 235--265.

\bibitem{Bouwer1972}
I.~Z. Bouwer, On edge but not vertex transitive regular graphs, \emph{J. Combinatorial Theory Ser.~B} 12 (1972), 32--40.

\bibitem{CL2024}
Q. Cai and Z.~P. Lu, Biprimitive edge-transitive pentavalent graphs, \emph{Discrete Math.} 347 (2024), no.~4, Paper No.~113871, 16~pp.


\bibitem{CMMP2006}
M. Conder, A. Malni\v{c}, D. Maru\v{s}i\v{c} and P. Poto\v{c}nik, A census of semisymmetric cubic graphs on up to 768 vertices, \emph{J. Algebraic Combin.} 23 (2006), no.~3, 255--294.

\bibitem{CV2019}
M.~D.~E. Conder and G. Verret, Edge-transitive graphs of small order and the answer to a 1967 question by Folkman, \emph{Algebr. Comb.} 2 (2019), no.~6, 1275--1284.

\bibitem{DM1996}
J.~D. Dixon and B. Mortimer, \emph{Permutation groups}, Springer-Verlag, New York, (1996).

\bibitem{DM19991}
S.-F. Du and D. Maru\v{s}i\v{c}, An infinite family of biprimitive semisymmetric graphs, \emph{J. Graph Theory} 32 (1999), no.~3, 217--228.

\bibitem{DM1999}
S.-F. Du and D. Maru\v{s}i\v{c}, Biprimitive graphs of smallest order, \emph{J. Algebraic Combin.} 9 (1999), no.~2, 151--156.

\bibitem{DX2000}
S. Du and M. Xu, A classification of semisymmetric graphs of order $2pq$, \emph{Comm. Algebra} 28 (2000), no.~6, 2685--2715.

\bibitem{FW2023}
R. Feng and L. Wang, Semisymmetric graphs of order twice prime powers with the same prime valency, \emph{J. Algebraic Combin.} 57 (2023), no.~4, 1285--1301.

\bibitem{Folkman1967}
J. Folkman, Regular line-symmetric graphs, \emph{J. Combinatorial Theory} 3 (1967), 215--232.

\bibitem{Guralnick1983}
R.~M. Guralnick, Subgroups of prime power index in a simple group, \emph{J. Algebra} (1983), no.~2, 304--311.

\bibitem{HL2015}
H. Han and Z. Lu, Semisymmetric graphs admitting primitive groups of degree $9p$, \emph{Sci. China Math.} 58 (2015), no.~12, 2671--2682.

\bibitem{HB1982}
B. Huppert and N. Blackburn, \emph{Finite Groups II}, Springer-Verlag, Berlin-New York, 1982.

\bibitem{II1985}
A. A. Ivanov and M. E. Iofinova, Biprimitive cubic graphs, in \emph{Investigations in the algebraic theory of combinatorial objects}, Institute for System Studies, Moscow, 1985, pp.~459--472.

\bibitem{Ivanov1987}
A. V. Ivanov, On edge but not vertex transitive regular graphs, \emph{Ann. Discrete Math.} 34 (1987).

\bibitem{Kantor1987}
W. M. Kantor, Primitive permutation groups of odd degree, and an application to finite projective planes, \emph{J. Algebra} 106 (1987), 15--45.

\bibitem{KLST1990}
W. Kimmerle, R. Lyons, R. Sandling and D.~N. Teague, Composition factors from the group ring and Artin's theorem on orders of simple groups, \emph{Proc. London Math. Soc. (3)} 60 (1990), no.~1, 89--122.

\bibitem{KL1990}
P.~B. Kleidman and M.~W. Liebeck, \emph{The subgroup structure of the finite classical groups}, Cambridge University Press, Cambridge, 1900.

\bibitem{LL2022}
G. Li and Z. P. Lu, Pentavalent semisymmetric graphs of square-free order, \emph{J. Graph Theory} 101 (2022), no.~1, 106--123.

\bibitem{LPS1987}
M.~W. Liebeck, C.~E. Praeger and J. Saxl, A classification of the maximal subgroups of the finite alternating and symmetric groups, \emph{J. Algebra} 111 (1987), 365--383.

\bibitem{LPS2000}
M.~W. Liebeck, C.~E. Praeger and J. Saxl, Transitive subgroups of primitive permutation groups, \emph{J. Algebra} (2000), no.~2, 291--361.

\bibitem{LPS2002}
M.~W. Liebeck, C.~E. Praeger and J. Saxl, Primitive permutation groups with a common suborbit, and edge-transitive graphs, \emph{Proc. London Math. Soc. (3)} 84 (2002), 405--438.

\bibitem{LS1985}
M.~W. Liebeck and J. Saxl, The primitive permutation groups of odd degree, \emph{J. London Math. Soc. (2)} 31 (1985), no.~2 250--264.

\bibitem{LX2002}
S. Lipschutz and M.-Y. Xu, Note on infinite families of trivalent semisymmetric graphs, \emph{European J. Combin.} 23 (2002), no.~6, 707--711.

\bibitem{Maroti2002}
A. Mar\'{o}ti, On the orders of primitive groups, \emph{J. Algebra} 258 (2002), 631--640.

\bibitem{MP2002}
D. Maru\v{s}i\v{c} and P. Poto\v{c}nik, Bridging semisymmetric and half-arc-transitive actions on graphs, \emph{European J. Combin.} 23 (2002), no.~6, 719--732.

\bibitem{Moll2012}
V.~H. Moll, \emph{Numbers and functions}, American Mathematical Society, Providence (2012), xxiv+504 pp.


\bibitem{Robbins1955}
H. Robbins, A remark on Stirling's formula, \emph{Amer. Math. Monthly} 62 (1955), 26--29.

\bibitem{RS1962}
J.~B. Rosser and L. Schoenfeld, Approximate formulas for some functions of prime numbers, \emph{Illinois J. Math.} 6 (1962), 64--94.

\bibitem{Stefan}
K. Stefan, A bound on the order of the outer automorphism group of finite simple group of given order, https://stefan-kohl.github.io/preprints/outbound.pdf.

\bibitem{WG2021}
L. Wang and S.-T. Guo, Semisymmetric prime-valent graphs of order $2p^3$, \emph{J. Algebraic Combin.} 54 (2021), no.~1, 49--73.

\bibitem{Wilson2003}
S. Wilson, A worthy family of semisymmetric graphs, \emph{Discrete Math.} 271 (2003), no.~1-3, 283--294.

\bibitem{YFX2023}
F.-G. Yin, Y.-Q. Feng and B. Xia, The smallest vertex-primitive $2$-arc-transitive digraph, \emph{J. Algebra} 626 (2023), 1--38.

\end{thebibliography}
\end{document}